\documentclass[11pt,a4paper]{amsart}
\usepackage{amssymb}
\usepackage{amsthm}
\usepackage{amsmath}
\usepackage{amsxtra}
\usepackage{latexsym}
\usepackage{mathrsfs}
\usepackage[all,cmtip]{xy}
\usepackage{enumitem}
\usepackage{xcolor}
\usepackage{graphicx}
\usepackage{comment}
\usepackage{mathabx,epsfig}
\usepackage{stmaryrd}
\usepackage{comment}
\usepackage{mathtools}
\usepackage{tikz-cd} 

\usepackage{hyperref}

\newcommand{\RNum}[1]{\uppercase\expandafter{\romannumeral #1\relax}}

\newtheorem{thm}{Theorem}[section]
\newtheorem{lem}[thm]{Lemma}

\newtheorem{claim}[thm]{Claim}
\newtheorem{prop}[thm]{Proposition}
\newtheorem{cor}[thm]{Corollary}

\theoremstyle{definition}
\newtheorem{defn}[thm]{Definition}
\newtheorem{rmk}[thm]{Remark}
\newtheorem{example}[thm]{Example}
\newtheorem{question}[thm]{Question}

\numberwithin{equation}{section}

\newcommand\be{\begin{equation}}
\newcommand\ba{\begin{eqnarray}}
\newcommand\ee{\end{equation}}
\newcommand\ea{\end{eqnarray}}

\def\C{{\mathbb C}}
\def\Q{{\mathbb Q}}

\def\Z{{\mathbb Z}}
\def\P{{\mathbb P}}
\def\A{{\mathbb A}}
\def\N{{\mathbb N}}

\def\O{{ \mathcal{O}}}

\def\k{{ \mathfrak{k}}}

\def\p{{\mathfrak{p}}}

\def\G{{\mathbb{G}}}

\DeclareMathOperator{\Prep}{Prep}

\DeclareMathOperator{\id}{id}
\DeclareMathOperator{\Aut}{Aut}

\DeclareMathOperator{\Orb}{Orb}

\DeclareMathOperator{\val}{val}
\DeclareMathOperator{\lcm}{lcm}

  {\end{list}}

\title[Dynamical GCD Problems and a variant of DML Conjecture]
{Dynamical GCD Problems and a Variant of the Dynamical Mordell-Lang Conjecture}
\thanks{S.Y. was supported by NNSFC grant No.~12271007. X.Z. was supported by NSERC grant RGPIN-2022-02951.}

\author{She Yang}
\address{ Peking University\\
 Beijing International Center for Mathematical Research\\
 Beijing\\
China 100871}
\email{ys-yx@pku.edu.cn}

\author{Xiao Zhong}
\address{University of Waterloo \\
Department of Pure Mathematics \\
Waterloo, Ontario \\
Canada  N2L 3G1}
\email{x48zhong@uwaterloo.ca}

\date{\today}
\begin{document}

\maketitle
\begin{abstract}
  In \cite{NZ25}, the authors resolved the rational function analogue of the finiteness results for greatest common divisors of iterates of polynomials established in \cite{HT17}. These results may be viewed as dynamical generalizations of a classical problem concerning upper bounds for the greatest common divisors (GCDs) of two integer sequences studied by Bugeaud, Corvaja, and Zannier. The most delicate case arises when the maps involved are automorphisms, where the methods of \cite{NZ25} and \cite{HT17} rely heavily on Diophantine approximation and asymptotic analysis.

In the present paper, we develop an alternative approach to the automorphism case. This method is more powerful, allowing us to give complete answers to the further questions posed in \cite{HT17}. In particular, we strengthen the main theorem of \cite{HT17} and provide an alternative proof of the main theorem of \cite{NZ25} in the automorphism setting.

Moreover, we relate this dynamical GCD problem to a special case of a higher-dimensional generalization of the Dynamical Mordell--Lang Conjecture proposed by Junyi Xie. We establish this generalized conjecture when the dynamics arise from algebraic group actions. In addition, we resolve the corresponding special case associated with dynamical GCD questions when the maps involved are polynomials.

\end{abstract}
\section{Introduction}
\subsection{Historical Background}

In 2003, Bugeaud, Corvaja, and Zannier \cite{BCZ03} established an upper bound for the greatest common divisors (GCD) of two integer sequences using classical tools from Diophantine approximation. More precisely, they proved that for any multiplicatively independent integers \( a,b \geq 2 \) and for any \( \epsilon > 0 \), one has
\[
\mathrm{gcd}(a^n - 1, b^n - 1) < \exp(\epsilon n)
\]
for all sufficiently large integers \( n \geq 1 \).

Shortly afterward, Ailon and Rudnick \cite{AR04} obtained the first analogue of such a GCD-type result over function fields of characteristic \( 0 \). They showed that for any nonconstant, multiplicatively independent polynomials \( a,b \in \mathbb{C}[x] \), there exists a polynomial \( h \in \mathbb{C}[x] \) such that
\[
\mathrm{gcd}(a^n - 1, b^n - 1) \mid h
\]
for all \( n \geq 1 \).  
In fact, the Ailon-Rudnick result holds more generally for all pairs \( (m,n) \in \mathbb{N} \times \mathbb{N} \), not only when \( m = n \).

Since then, numerous generalizations and extensions of these GCD-type results have been established in various settings; see, for example, \cite{CZ13}, \cite{GHT17}, \cite{GHT18}, \cite{Gr20}, \cite{GW20}, \cite{Hu20}, \cite{Le19}, \cite{Lu05}, \cite{Si04a}, \cite{Si04b}, \cite{Si05}, and \cite{Za12}.

Among these developments, a particularly interesting dynamical analogue was first proposed by Ostafe \cite[\S4.2]{Os16}, who asked whether one could obtain a similar bound on the degree of greatest common divisors between two sequences of polynomials constructed via iteration rather than multiplicative powers.  
Subsequently, in 2017, Hsia and Tucker \cite{HT17} provided the first instance of such a GCD-type result in a dynamical setting.

Let \( f(x) \in \mathbb{C}(x) \) be a rational function with complex coefficients, and denote by \( f^n \) its \( n \)-fold composition.

\begin{defn}
Two rational functions \( f \) and \( g \) are said to be \emph{compositionally independent} if, whenever
\[
f^{i_1} \circ g^{j_1} \circ \cdots \circ f^{i_s} \circ g^{j_s}
= f^{l_1} \circ g^{m_1} \circ \cdots \circ f^{l_t} \circ g^{m_t},
\]
for some positive integers \( i_k, j_k, l_k, m_k \), it follows that \( s = t \) and \( i_k = l_k, j_k = m_k \) for all \( 1 \leq k \leq s \).  
Equivalently, the semigroup generated by \( f \) and \( g \) under composition is isomorphic to the free semigroup on two generators.
\end{defn}

Hsia and Tucker proved that if \( f, g \in \mathbb{C}[x] \) are polynomial maps of degree at least \( 2 \), \( c(x) \in \mathbb{C}[x] \), \( f \) and \( g \) are compositionally independent, and \( c \) is not an iterate of either \( f \) or \( g \), then the polynomial
\[
\mathrm{gcd}(f^m(x) - c(x), \, g^n(x) - c(x))
\]
has only finitely many possible irreducible factors across all positive integers \( m, n \) (cf. \cite[Theorem~1]{HT17}). Moreover, when one requires that $m = n$, their finiteness result of the number of possible irreducible factors also holds when $f$ and $g$ are allowed to be automorphisms.

More recently, Noytaptim and the second author \cite{NZ25} extended Hsia and Tucker's result to rational functions and addressed the following question originally proposed in \cite{HT17}:

\begin{question}[\cite{HT17}, Question~18] \label{qu: HTquestion}
Let \( f,g \) be two nonconstant, compositionally independent rational functions with complex coefficients, and let \( c \in \mathbb{C}(x) \).  
Is it true that there exist at most finitely many \( \lambda \in \mathbb{C} \) such that
\[
f^n(\lambda) = g^n(\lambda) = c(\lambda)
\]
for some positive integer \( n \)?
\end{question}

They answered this question affirmatively when at least one of \( f \) or \( g \) has degree greater than \( 1 \). Moreover, they proved that when both \( f \) and \( g \) are automorphisms of \( \mathbb{P}^1 \), the expected finiteness holds after excluding certain explicitly classified exceptional families of triples \( (f,g,c) \) (see \cite{NZ25}).

In subsequent work, they also address some special cases of \cite[Question 19]{HT17} which explore higher-dimensional generalizations of Question~\ref{qu: HTquestion}. In \cite{NZ24}, Noytaptim and the second author studied such extensions for special families of endomorphisms, including H\'enon maps, split endomorphisms on \( (\mathbb{P}^1)^n \), and regular skew products of polynomials on \( \mathbb{A}^2 \).
\subsection{A New Approach to the Automorphisms Cases}
In \cite[Remark~(1)]{HT17}, Hsia and Tucker noted that they had no clear approach to the problem when replacing
\[
\gcd(f^m(x) - c(x),\, g^n(x) - c(x))
\]
by
\[
\gcd(f^m(x) - c_1(x),\, g^n(x) - c_2(x)),
\]
where \( c_1 \) and \( c_2 \) are distinct rational functions and \( f, g \) are linear polynomials.  
Indeed, there is little reason to expect that their finiteness conclusion-stating that compositionally independent \( f \) and \( g \) yield only finitely many linear factors in the GCD-continues to hold in this setting.  
The results of \cite{NZ25} already suggest a fundamental difference between the cases of automorphisms and that of higher-degree endomorphisms.

More generally, it is natural to ask the following question without restricting to polynomial maps.

\begin{question}\label{qu: automorphisms-common-zeros}
Let \( K \) be an algebraically closed field.  
Let \( f \) and \( g \) be automorphisms of \( \mathbb{P}^1 \) defined over \( K \).  
For which pairs of rational functions \( c_1, c_2 \in K(x) \), depending on \( f \) and \( g \), do there exist infinitely many \( \lambda \in K \) such that
\[
f^m(\lambda) = c_1(\lambda) \quad \text{and} \quad g^n(\lambda) = c_2(\lambda)
\]
for some \( m, n \in \mathbb{N}^+ \)?
\end{question}

In the first part of this paper, we provide a complete answer to this question. 
Under the assumption that $c_1$ and $c_2$ are not iterates of $f$ and $g$, respectively, 
we summarize our classification results below, with the sole exception of the 
characteristic $0$ case in which both $f$ and $g$ are translation maps. 
In that situation, the problem reduces to a classification of Siegel curves 
admitting infinitely many integral points.

Note that, up to simultaneous conjugation of $f$ and $g$ by an automorphism, 
we may assume that
\[
f(x) = \alpha x + \beta,
\]
and that either
\[
g(x) = \delta x + \gamma
\quad \text{or} \quad
g(x) = \frac{x}{\gamma x + \delta},
\]
depending on whether $f$ and $g$ share a common fixed point.

When $\operatorname{char}(K) = 0$, the following two theorems describe all pairs 
$(c_1, c_2)$ sought in Question~\ref{qu: automorphisms-common-zeros}, up to conjugation.

\begin{thm}\label{thm: main-classification-common-root}
    Let $f(x) = \alpha x + \beta$ and $g(x) = \delta x + \gamma$, where $\alpha, \delta \in \C^*$ and $\beta, \gamma \in \C$. Let $c_1$ and $c_2$ be two rational functions such that $c_1 \not\in \{f^n:n \in \N^+\}$ and $c_2 \not\in \{g^n:n\in\N^+\}$. Then there are infinitely many $\lambda \in \C$ such that 
    $$ f^m(\lambda) = c_1(\lambda)$$
    $$ g^n(\lambda) = c_2(\lambda)$$
    holds for some $m,n \in \N^+$ implies the following holds:
    \begin{enumerate}
        \item If $\alpha$ and $\delta$ are not $1$, then 
         $$ c_1(x) = \mu F(x)^p\left (x - \frac{\beta}{1-\alpha}\right) + \frac{\beta}{1 - \alpha}$$
    $$c_2(x) = F(x)^q\left(x - \frac{\gamma}{1 - \delta}\right) + \frac{\gamma}{1 - \delta},$$
    where $F(x)$ is a non-constant rational function, $p,q$ are non-zero coprime integers and $\mu \in \C^*$ such that $\mu^q= \alpha^{mq}/\delta^{np}$ for infinitely many pairs of $(m,n) \in (\N^+)^2$.
    \item If $\alpha \in \C^* \setminus \{1\}$ and $\delta = 1$, then $\alpha \in \overline{\Q}^*$ and there exist non-constant rational functions $F(x)$ and $B(x)$ and $d \in \Z$ such that 
    $$ c_1(x) =  F(x)^d\left(x - \frac{\beta}{1- \alpha}\right) + \frac{\beta}{1 - \alpha}$$
    $$ c_2(x) = \gamma B(F(x)) + x.$$
    Moreover, $B(x) \in \overline{\Q}[x] $ and either $\alpha$ or $1/\alpha$ is an algebraic integer.
    \end{enumerate}
\end{thm}
\begin{thm}\label{thm: main-classification-no-common-root}
    Let $f(x) = \alpha x + \beta$ and $g(x) = x/(\gamma x + \delta)$, where $\alpha, \delta, \beta, \gamma \in \C^*$. Let $c_1$ and $c_2$ be two rational functions such that $c_1 \not\in \{f^n:n \in \N^+\}$ and $c_2 \not\in \{g^n:n\in\N^+\}$. Then there are infinitely many $\lambda \in \C$ such that 
    $$ f^m(\lambda) = c_1(\lambda)$$
    $$ g^n(\lambda) = c_2(\lambda)$$
    holds for some $m,n \in \N^+$ implies the following holds:
    \begin{enumerate}
        \item If $\alpha$ and $\delta$ are not $1$, then $$ c_1(x) = \left(x- \frac{\beta}{1 - \alpha}\right)\mu F(x)^p + \frac{\beta}{1 - \alpha}$$
        $$ c_2(x) = \frac{x}{F(x)^q (1 - \gamma x/(1- \delta)) + \gamma /(1 - \delta)},$$
         where $F(x)$ is some non-constant rational function, $p,q$ are some coprime non-zero integers and $\mu \in \C^*$ such that $\mu^q= \alpha^{mq}/\delta^{np}$ for infinitely many pairs of $n,m \in \N^+$.
         \item If $\delta \in \C^* \setminus \{1\}$ and $\alpha = 1$, then $\delta \in \overline{\Q}^*$ and there exist non-constant rational functions $F(x)$ and $B(x)$ and $d \in \Z$ such that 
    $$ c_1(x) = \beta B(F(x)) + x$$
    $$ c_2(x) = x/\left((1 - F(x)^d)\frac{\gamma}{1 - \delta}x + F(x)^d\right).$$
    Moreover, $B(x) \in \overline{\Q}[x] $ and either $\delta$ or $1/\delta$ is an algebraic integer.
    \end{enumerate}
\end{thm}

\begin{rmk}
    The statements above provide only a partial summary of the classification 
results established in Section~\ref{sec: common-auto-case} and do not necessarily reflect 
their full strength. A more detailed analysis will be given in 
Section~\ref{sec: common-auto-case}.
\end{rmk}

When $\operatorname{char}(K) >0$, we obtain the following classifications up to conjugations: 

\begin{thm}\label{thm: main-classification-positive-char}
Let $f(x)=\alpha x+\beta$ and $g(x)=\delta x+\gamma$, where $\alpha,\delta\in K\backslash\overline{\mathbb{F}_p}$ and $\beta,\gamma\in K$. Let $c_1(x),c_2(x)\in K(x)$ be rational functions such that $c_1 \not\in \{f^n : n \in \Z\}$ and $c_2 \not \in \{g^n : n \in \Z\}$. Write $$C_1(x)=\frac{c_1(x)-\frac{\beta}{1-\alpha}}{x-\frac{\beta}{1-\alpha}}$$ and $$C_2(x)=\frac{c_2(x)-\frac{\gamma}{1-\delta}}{x-\frac{\gamma}{1-\delta}}.$$

\begin{enumerate}
\item
$\bigcup\limits_{n\in\mathbb{Z}}\left\{x\in K|\ f^n(x)=c_1(x),g^n(x)=c_2(x)\right\}$ is an infinite set if and only if $C_1$ and $C_2$ are non-constant, and there exist
\begin{enumerate}
\item
$e_1\in\mathbb{Z},e_2\in\mathbb{Z}\backslash\{0\}$, and $q$ which is a power of $p$,
\item
$\alpha',\delta'\in K$ such that $\alpha'^{q-1}=\alpha$ and $\delta'^{q-1}=\delta$, and
\item
an absolutely irreducible polynomial $F(x,y)\in\mathbb{F}_q[x,y]$,
\end{enumerate}
such that $F(\alpha'^{e_2},\delta'^{e_2})=0$ and $$F\left(\frac{\alpha'^{e_2}}{\alpha^{e_1}}\cdot C_1(x),\frac{\delta'^{e_2}}{\delta^{e_1}}\cdot C_2(x)\right)$$ is identically zero.

\item
$\bigcup\limits_{(m,n)\in\mathbb{Z}^2}\left\{x\in K|\ f^m(x)=c_1(x),g^n(x)=c_2(x)\right\}$ is an infinite set if and only if $C_1$ and $C_2$ are non-constant, and there exist
\begin{enumerate}
\item
$(m_0,n_0)\in\mathbb{Z}^2,(m_1,n_1)\in\mathbb{Z}^2\backslash\{(0,0)\}$, and $q$ which is a power of $p$,
\item
$\alpha',\delta'\in K$ such that $\alpha'^{q-1}=\alpha$ and $\delta'^{q-1}=\delta$, and
\item
an absolutely irreducible polynomial $F(x,y)\in\mathbb{F}_q[x,y]$
\end{enumerate}
such that $F(\alpha'^{m_1},\delta'^{n_1})=0$ and $$F\left(\frac{\alpha'^{m_1}}{\alpha^{m_0}}\cdot C_1(x),\frac{\delta'^{n_1}}{\delta^{n_0}}\cdot C_2(x)\right)$$ is identically zero.
\end{enumerate}

If $$g(x)=\frac{x}{\gamma x+\delta},$$ then we write $$C_2(x)=x\cdot\frac{\frac{1}{c_2(x)}-\frac{\gamma}{1-\delta}}{1-\frac{\gamma x}{1-\delta}}$$ and the same statements hold.
\end{thm}

When \( \operatorname{char}(K) = 0 \), our main observation is that any such pair \( (c_1, c_2) \), if it exists, parametrizes a genus-\(0\) curve containing infinitely many points of the form
\[
\{ (\alpha^m, \delta^n) : m,n \in \mathbb{N}^+ \}
\]
for some \( \alpha, \delta \in \mathbb{C}^* \).  
Using the established Mordell-Lang conjecture, we are then able to determine \( c_1 \) and \( c_2 \) explicitly in terms of the coefficients of \( f \) and \( g \).

When \( \operatorname{char}(K) > 0 \), our classification provides a counterexample to the following question posed by Hsia and Tucker:

\begin{question} \label{qu: positive-char}[First part of \cite{HT17}, Question~17]
Let \( f \) and \( g \) be two compositionally independent, non-isotrivial polynomials in \( K[x] \), and let \( c \in K[x] \).  
Is it true that there exist only finitely many \( \lambda \in K \) such that, for some \( n \in \mathbb{N}^+ \),
\[
(x - \lambda) \mid \gcd(f^n(x) - c(x),\, g^n(x) - c(x))?
\]
\end{question}
\setlength{\parskip}{3pt}
\subsubsection{Further applications in characteristic $0$} At the beginning of Section~3 in \cite{HT17}, Hsia and Tucker provided an example with
\[
f(x) = 2x, \quad g(x) = x + 1, \quad c(x) = x^2,
\]
showing that there are infinitely many \( \lambda \in \mathbb{C} \) for which
\[
(x - \lambda) \mid \gcd(f^m(x) - c(x),\, g^n(x) - c(x))
\]
for some \( m, n \in \mathbb{N}^+ \).  
Thus, it appears difficult to expect a general finiteness statement in this case.  
However, in their example, \( f \) and \( g \) are \emph{not} compositionally independent.  
Using our new approach, we can now establish the following finiteness result which strengthen \cite[Theorem 1]{HT17}.

\begin{thm}\label{thm: f^m-g^n-c-poly}
Let \( f, g, c \in \mathbb{C}[x] \) be polynomials such that \( c(x) \) is not an iterate of either \( f \) or \( g \).  
Suppose that \( f \) and \( g \) are compositionally independent.  
Then there exist only finitely many \( \lambda \in \mathbb{C} \) such that
\[
f^m(\lambda) = g^n(\lambda) = c(\lambda)
\]
for some \( m, n \in \mathbb{N}^+ \).
\end{thm}

Moreover, our approach provides a new proof of the result of \cite{NZ25} concerning automorphism cases.

\begin{thm} \label{thm: mainaut}
    Let $f(x)$, $g(x)$ be automorphisms on $\P^1$ defined over $\C$ and $c(x)$ be a rational function defined over $\C$. If the semigroup generated by $f(x)$ and $g(x)$ under compositions is free, then there are only finitely many $\lambda \in \C$ such that 
    \begin{equation} 
        c(\lambda) = f^n(\lambda) = g^n(\lambda)
    \end{equation}
    for some positive integer $n$ unless $f(x)$, $g(x)$ are simultaneously conjugated by an automorphism on $\P^1(\C)$ to one of the following:
    \begin{enumerate}
        \item $\alpha x + \beta$, $x/(\gamma x + \delta)$, with some $\alpha, \delta, \gamma, \beta \in \C^*$ such that one of $\alpha /\delta $ and $\alpha \delta$ is a root of unity;
        \item $\alpha x + \beta$, $\delta x + \gamma$,  with some $\alpha, \delta \in \C^*$ such that $\alpha$ and $\delta$ are not roots of unity, $\gamma$ and $\beta$ are not both $0$, and either $\alpha/\delta$ is a root of unity other than $1$ or one of $\alpha^2/\delta$ and $\alpha/
        \delta^2$ is a root of unity.
    \end{enumerate}
    
\end{thm}

In the original version of \cite{NZ25}, a minor technical error in the proof led to the omission of one exceptional case.  
This oversight is corrected in (\cite{NZ25b}).
\subsection{Connection to the Dynamical Mordell--Lang Conjecture}

For an algebraic dynamical system, the dynamical Mordell--Lang Conjecture captures the phenomenon of \emph{unlikely intersections} between the orbit of a point and a subvariety.  
One may naturally expect a similar phenomenon when considering the orbit of a subvariety intersecting another subvariety of small dimension.  
To this end, Junyi Xie proposed the following question in \cite[Question~9.13(i)]{Xie23}.

\begin{question}\label{qu: DML-varieties}
Let \( X \) be a variety and \( f \) an endomorphism of \( X \).  
Let \( Z \) and \( V \) be irreducible closed subvarieties of \( X \) such that \( \dim(Z) + \dim(V) < \dim(X) \).  
Can we describe the set
\[
\{n \in \mathbb{N} \mid f^n(Z) \cap V \neq \emptyset\}?
\]
If the characteristic of the base field is \( 0 \), is this set a finite union of arithmetic progressions?
\end{question}

Here and throughout, we interpret \( f^n(Z) \) as the set-theoretic image of \( Z \) under the \( n \)-fold iterate of \( f \).

Somewhat surprisingly, even in characteristic \( 0 \), this set can exhibit highly nontrivial behavior.  
In a private communication with Junyi Xie, Jungin Lee and Gyeonghyeon Nam provided the following striking example.

 \begin{example}(Lee--Nam)
 Let $X=\mathbb{A}_{\mathbb{C}}^5$. Let $Z=\{(x_1,x_2,x_3,x_4,x_5)|\ x_1=0,x_2=x_4=1\}$ and $V=\{(x_1,x_2,x_3,x_4,x_5)|\ x_2=x_4=0,x_1=(x_3+1)(x_5+1)\}$. Let $f$ be the endomorphism of $X$ given by $f(x_1,x_2,x_3,x_4,x_5)=(x_1+1,x_2(x_3-x_1-1),x_3,x_4(x_5-x_1-1),x_5)$. Then $\{n\in\mathbb{N}|\ f^n(Z)\cap V\neq\emptyset\}$ is the set of composite numbers.
 \end{example}

As this example shows, even in characteristic zero, a complete description of such sets can be elusive.  
Nevertheless, it remains a natural and compelling problem to investigate special situations in which the structure of this set can be explicitly determined.\\

The problem discussed earlier in this paper can, in fact, be viewed as a special case of this general question.

Let \( X \) be an irreducible curve defined over \( \mathbb{C} \), and let \( f, g, c : X \to X \) be endomorphisms.  
Define
\[
\Delta \coloneqq \{ (x,x,x) : x \in X \} \subseteq X^3, \qquad
C \coloneqq \{ (x,x,c(x)) : x \in X \} \subseteq X^3.
\]
Let \( F \coloneqq (f, g, \operatorname{Id}) : X^3 \to X^3 \) be the endomorphism acting componentwise.

Then the set of \( n \in \mathbb{N} \) for which there exists \( x \in X \) satisfying
\[
f^n(x) = g^n(x) = c(x)
\]
is precisely
\[
\{ n \in \mathbb{N} \mid F^n(C) \cap \Delta \neq \emptyset \}.
\]
Thus, describing this set of integers \( n \) can be seen as a special case of Question~\ref{qu: DML-varieties}.

In the second half of this paper, we show that when \( X = \mathbb{P}^1 \) and \( f, g \) are automorphisms of \( \mathbb{P}^1 \), this set is a finite union of arithmetic progressions.  
More generally, we prove the same conclusion for endomorphisms arising from algebraic group actions on projective varieties and also describe the structure of this set when the maps are defined over a field of positive characteristic:

\begin{thm} \label{thm: main-group-GDML}
Let $f$ be an automorphism of a projective variety $X$. Suppose that $f$ comes from an algebraic group action. Let $Z$ and $V$ be closed subvarieties of $X$. Consider the set $\{n\in\mathbb{Z}|\ f^n(Z)\cap V\neq\emptyset\}$.
\begin{enumerate}
\item
If $\mathrm{char}(K)=0$, then this set is a finite union of arithmetic progressions.
\item
If $\mathrm{char}(K)=p>0$, then this set is a widely $p$-normal set.
\end{enumerate}
\end{thm}

Moreover, beyond the context of endomorphisms induced by algebraic group actions, we establish the same conclusion in the case when \( X = \mathbb{A}^1_{\mathbb{C}} \) and \( f, g, c \) are polynomials defined over \( \mathbb{C} \):
\begin{thm}\label{thm: main-GDML-polynomial}
    Let $f$ and $g$ be non-constant polynomial defined over $\C$ and $c$ be a polynomial defined over $\C$. Then the set of $n \in \N$ such that 
    \begin{equation}
        f^n(x) = g^n(x) = c(x)
    \end{equation} 
    admits solution is a finite union of arithmetic progressions.
\end{thm}

This case are tractable because the corresponding dynamical behaviors are well understood. Our proof relies crucially on the classification theorem of \cite{SS95}, which characterizes polynomials sharing the same set of preperiodic points. Also, it turns out that the power maps case connects tightly to a subtle elementary number theory problem concerning greatest common divisors which we discussed and resolved in the preliminary subsection of Section \ref{sec: GDML-poly}. We put here the statement to highlight it in case it is of independent interests for some reader: 

\begin{prop}\label{prop: gcd-arithmetic-progressions}
Let \( K \) be a number field with ring of integers \( \mathcal{O} \).  
Let $\k $ be an ideal in $\O$, and let \( d_1, d_2 \in \mathcal{O} \).  
Let \( d_3, a, b \in \mathcal{O} \).  
Then the set
\[
\left\{\, n \in \mathbb{N} :\; b(d_1^{n} - d_3) - a(d_2^{n} - d_3) \in 
\k\, (d_1^{n} - d_3,\; d_2^{n} - d_3) \,\right\}
\]
is a finite union of arithmetic progressions.
\end{prop}
\begin{rmk}
   We expect that this result can be further generalized to hold over any Noetherian integral domain of characteristic \(0\) and Krull dimension one.
\end{rmk}

\begin{rmk}
    When $K = \Q$, Proposition \ref{prop: gcd-arithmetic-progressions} implies that the set of $n\in \N$ such that 
    $$ f^n (x) = g^n(x) = c(x)$$
    admits a solution in $\C^*$ is a finite union of arithmetic progressions, where $f,g$ and $c$ are endomorphisms on $\C^*$. See the discussion in Section \ref{sec: GDML-poly} and Lemma \ref{lem: power-reduce-to-gcd}.
    \end{rmk}

\subsection{Outline of the paper}

In Section~\ref{sec: common-auto-case}, we develop a new approach to the automorphism case of 
Question~\ref{qu: HTquestion} and obtain a series of classification results addressing the more general 
Question~\ref{qu: automorphisms-common-zeros}. As an application of this classification, we prove 
Theorem~\ref{thm: f-g-m-n-linear-finite}, which strengthens \cite[Theorem~1]{HT17}. We also provide 
an alternative proof of \cite[Theorem~1.3]{NZ25} using our new method. In addition, when the base 
field has positive characteristic, our classification yields a counterexample to the first part 
of \cite[Question~17]{HT17}.

In Section~\ref{sec: connect-GDML-group}, we relate our dynamical GCD problem to a special case of 
the generalized Dynamical Mordell--Lang question (Question~\ref{qu: DML-varieties}). We resolve 
this question in the setting where the dynamics arise from algebraic group actions.

In Section~\ref{sec: GDML-poly}, we study the special case of Question~\ref{qu: DML-varieties} that 
is directly connected to our dynamical GCD framework. In this setting, the problem is equivalent 
to describing the set of integers \( n \in \mathbb{N} \) for which the equations
\[
f^n(x) = g^n(x) = c(x)
\]
admit common solutions in \( \mathbb{A}^1_{\mathbb{C}} \). We completely resolve this problem when 
\( f \), \( g \), and \( c \) are polynomials.

\section{Common Zeros of Iterated M\"obius Transformations }\label{sec: common-auto-case}
In this section, we first answer Question~\ref{qu: automorphisms-common-zeros} using our new approach, dividing the discussion into several propositions that treat different cases. We then focus on the system \( \{f, g, c\} \), where \( f \) and \( g \) are M\"obius transformations on \( \mathbb{P}^1 \) and \( c \) is a rational function. In this setting, we prove Theorem~\ref{thm: f^m-g^n-c-poly} and provide an alternative proof of Theorem~\ref{thm: mainaut}. Finally, we consider the case where the base field has positive characteristic and address Question~\ref{qu: positive-char}.

\subsection{On Question \ref{qu: automorphisms-common-zeros}}
In this section, we give a complete characterization of when there exist infinitely many points \( \lambda \in \mathbb{C} \) satisfying  
\[
f^n(\lambda) = c_1(\lambda) \quad \text{and} \quad g^m(\lambda) = c_2(\lambda)
\]
for some \( m, n \in \mathbb{N}^+ \), where \( f \) and \( g \) are M\"obius transformations except the case that $f$ and $g$ are both translation maps (see the discussion at the end of the subsection).  
Throughout this subsection, we assume that the pairs of rational functions \( \{f, c_1\} \) and \( \{g, c_2\} \) are \textbf{non-degenerate}, meaning that \( c_1 \) (respectively \( c_2 \)) is not equal to \( f^n \) (respectively \( g^n \)) for any $n \in \N^+$. The degenerate cases are elementary but somewhat tedious to verify, and we leave the details to the interested reader.

Note that \( f \) and \( g \) can be simultaneously conjugated to one of the following two forms:
\begin{itemize}
    \item \( f(x) = \alpha x + \beta \) and \( g(x) = \delta x + \gamma \), where \( \alpha, \delta \in \mathbb{C}^* \) and \( \beta, \gamma \in \mathbb{C} \);
    \item \( f(x) = \alpha x + \beta \) and \( g(x) = x/(\gamma x + \delta) \), where \( \alpha, \delta \in \mathbb{C}^* \) and \( \beta, \gamma \in \mathbb{C} \).
\end{itemize}
In what follows, we will discuss these two cases separately through a series of propositions.

\begin{prop}\label{prop: poly-m-n-not-root-case}
    Suppose $f(x) = \alpha x + \beta$ and $g(x) = \delta x + \gamma$, where $\alpha, \delta \in\C^*$ are not $1$. Let $c_1(x)$ and $c_2(x)$ be two rational functions. Then there are infinitely many $\lambda \in \C$ such that $$f^m(\lambda) = c_1(\lambda)$$
    $$ g^n(\lambda) = c_2(\lambda)$$
    holds for some $m,n \in \N^+$ implies the following holds 
    
         $$ c_1(x) = \mu F(x)^p\left (x - \frac{\beta}{1-\alpha}\right) + \frac{\beta}{1 - \alpha}$$
    $$c_2(x) = F(x)^q\left(x - \frac{\gamma}{1 - \delta}\right) + \frac{\gamma}{1 - \delta},$$
    where $F(x)$ is a non-constant rational function, $p,q$ are non-zero coprime integers and $\mu \in \C^*$ such that $\mu^q= \alpha^{mq}/\delta^{np}$ for infinitely many pairs of $(m,n) \in (\N^+)^2$.

\end{prop}
\begin{proof}
    We first note that if $\alpha$ is a root of unity other than $1$, then having infinitely many $\lambda \in \C$ such that $f^m(\lambda) = c_1(\lambda)$ for some $m \in \N^+$ will imply that $c_1 = f^m$ for some $m \in \N^+$ as $\{f^m : m \in \N^+\}$ is a finite set. This contradicts that $\{f,c_1\}$ is non-degenerate. So, we can assume that $\alpha$ is not a root of unity. For the same reason, we can also assume $\delta$ is not a root of unity.     
    
    Note that having infinitely many $\lambda \in \C$ such that $f^m(\lambda) = c_1(\lambda)$ and $g^n(\lambda) = c_2(\lambda)$ hold for some $m,n \in \N^+$ is equivalent to having infinitely many $\lambda$ such that
    \begin{equation}\label{eq-poly-m-n-1}
        \alpha^m \lambda + \frac{1 - \alpha^m}{1 - \alpha} \beta = c_1(\lambda)
    \end{equation}
    \begin{equation}\label{eq-poly-m-n-2}
        \delta^n \lambda + \frac{1 - \delta^n}{1 - \delta}\gamma = c_2(\lambda)
    \end{equation}
    hold for some $m,n \in \N^+$. Since we assumed that $\{f,c_1\}$ and $\{g,c_2\}$ are non-degenerate, we have these imply that
    the set of $(m,n) \in (\N^+)^2$ such that \begin{equation}
        \alpha^m = \frac{c_1(\lambda) - \beta/(1 - \alpha)}{\lambda - \beta/(1 - \alpha)}
    \end{equation}
    \begin{equation}
        \delta^n = \frac{c_2(\lambda) - \gamma/(1 - \delta)}{\lambda - \gamma/(1 - \delta)}
    \end{equation}
hold for some $\lambda \in \C$ are infinite and, moreover, its projection to both coordinates are infinite sets. Let $S \subseteq (\N^+)^2 $ denote this set of $(m,n)$. These are equivalent to say that the rational map $\A^1 \to \G^2_m$ given by
    \begin{equation}\label{eq: para-m-n-1}
        x \to \left(\frac{c_1(x) - \beta/(1 - \alpha)}{x - \beta/(1 - \alpha)} , \frac{c_2(x) - \gamma/(1 - \delta)}{x - \gamma/(1 - \delta)}\right)
    \end{equation} 
    parametrized a curve, $C$, contains $\{(\alpha^m, \delta^n) : (m,n)\in S\}$ and, in particular, it is not a vertical or horizontal line.

     Then by the Mordell-Lang Conjecture solved in \cite{MC95}, we have that there exists a pair of coprime integers $(p,q)$ and a $\mu' \in \C$ such that
    $$ C  = V(y^p - \mu' x^q) \subseteq \C^2,$$
    or equivalently, there exists a non-constant rational function $F(x) \in \C(x)$ such that the parametrization (\ref{eq: para-m-n-1}) satisfies that 
    $$ \frac{c_1(x) - \beta/(1 - \alpha)}{x - \beta/(1 - \alpha)} = \mu F(x)^p,$$
    $$ \frac{c_2(x) - \gamma/(1 - \delta)}{x - \gamma/(1 - \delta)} = F(x)^q ,$$
    
    for the same pair of $p,q \in \Z^*$ and $\mu^q= \alpha^{mq}/\delta^{np}$ for any $(n,m) \in S$.

    Thus, there are infinitely many $\lambda \in \C$ such that Equation (\ref{eq-poly-m-n-1}) and (\ref{eq-poly-m-n-2}) hold for some $m,n \in \N^+$ implies that 
    $$ c_1(x) = \mu F(x)^p\left(x - \frac{\beta}{1 - \alpha}\right) + \frac{\beta}{1 - \alpha}$$
    $$ c_2(x) = F(x)^q \left(x - \frac{\gamma}{1 - \delta}\right) + \frac{\gamma}{1 - \delta},$$
    where $F(x)$ is any rational functions and $\mu$ is as above.
\end{proof}

\begin{prop}\label{prop: non-poly-m-n-root}
    Suppose $f(x) = \alpha x + \beta$ and $g(x) = x/(\gamma x + \delta)$, where $\alpha, \delta \in \C^*$ is not $1$ and $\gamma \in \C^*$. Let $c_1(x)$ and $c_2(x)$ be two rational functions. Then there are infinitely many $\lambda \in \C$ such that $$f^m(\lambda) = c_1(\lambda)$$
    $$ g^n(\lambda) = c_2(\lambda)$$
    holds for some $m,n \in \N^+$ implies that the following holds 
    $$ c_1(x) = \left(x- \frac{\beta}{1 - \alpha}\right)\mu F(x)^p + \frac{\beta}{1 - \alpha}$$
        $$ c_2(x) = \frac{x}{F(x)^q (1 - \gamma x/(1- \delta)) + \gamma /(1 - \delta)},$$
         where $F(x)$ is some non-constant rational function, $p,q$ are some coprime non-zero integers and $\mu \in \C^*$ such that $\mu^q= \alpha^{mq}/\delta^{np}$ for infinitely many pairs of $n,m \in \N^+$.
    
\end{prop}
\begin{proof}
 Similarly as in the proof of Proposition \ref{prop: poly-m-n-not-root-case}, we can assume that $\alpha$ and $\delta$ are not roots of unity.
   
   Then, following the same argument as in Proposition \ref{prop: poly-m-n-not-root-case}, having infinitely many $\lambda$ such that 
    $$ \alpha^m \lambda + \frac{1 - \alpha^m}{1 - \alpha} \beta = c_1(\lambda)$$
    $$ \frac{\lambda}{\gamma \lambda (1 - \delta^n)/(1 - \delta) + \delta^n} = c_2(\lambda)$$
    for some $m,n\in \N^+$
    is equivalent to that the set of $(m,n)\in (\N^+)^2$ such that 
    $$ \alpha^m = \frac{c_1(\lambda) - \beta/(1 - \alpha)}{\lambda - \beta/(1 - \alpha)}$$
    $$ \delta^n = \frac{\lambda}{1- \gamma \lambda /(1 - \delta)}\left(1/c_2(\lambda) - \frac{\gamma}{1 - \delta}\right)$$
    hold for some $\lambda \in \C$ is infinite and its projection to both coordinates are infinite as we assumed $\{f,c_1\}$ and $\{g,c_2\}$ are non-degenerate. Let us denote this set as $S \subseteq (\N^+)^2$.
    
    Then, we have that the rational map from $\A^1 \to \G_m^2$
    \begin{equation}\label{eq: para-nonpoly-m-n-1}
        x \to \left(\frac{c_1(x) - \beta/(1- \alpha)}{x - \beta/(1- \alpha)}, \frac{x}{1- \gamma x/(1 - \delta)}\left(1/c_2(x) - \frac{\gamma}{1 - \delta}\right)\right) 
    \end{equation} 
 parameterizes a curve $C \subseteq \G^2_m$ contains $\{(\alpha^m,\delta^n) : (m,n) \in S\}$. By the Mordell-Lang Conjecture \cite{MC95},
    $$ C =V(y^p - \mu' x^q)$$
    for a pair of coprime integers $p,q$ and $\mu' \in \C^*$
    for any $(m,n) \in S$.

    Thus, there exists a non-constant rational function $F(x) \in \C(x)$ such that the parametrization (\ref{eq: para-nonpoly-m-n-1}) satisfies that
    $$\frac{c_1(x) - \beta/(1- \alpha)}{x - \beta/(1- \alpha)} = \mu F(x)^p, $$
    $$\frac{x}{1- \gamma x/(1 - \delta)}\left(1/c_2(x) - \frac{\gamma}{1 - \delta}\right) = F(x)^q ,$$
    where $\mu^q = \alpha^{mq}/\delta^{np}$ for any $(m,n) \in S$. Therefore, we have
    $$ c_1(x) = \mu F(x)^p \left(x - \frac{\beta}{1 - \alpha}\right) + \frac{\beta}{1 - \alpha},$$
    $$ c_2(x) = \frac{x}{F(x)^q(1 - \gamma x/(1 -\delta)) + \gamma x/(1 - \delta)} .$$

\end{proof}
\begin{rmk}\label{rmk: n=m-no-degen}
   
We also note that Proposition \ref{prop: poly-m-n-not-root-case} and \ref{prop: non-poly-m-n-root} remain valid even without assuming that the pairs \(\{f, c_1\}\) and \(\{g, c_2\}\) are non-degenerate, provided that we impose the condition \(n = m \in \N^+\), $c_1 = c_2$ and $\alpha, \delta, \alpha/\delta$ are not roots of unity. To see this, we only need to consider the exceptional case that there are infinitely many $\lambda \in \C$ and a $n \in \N$ such that 
$$ f^n(\lambda) = g^n(\lambda) = c(\lambda).$$
Under the assumption of proposition \ref{prop: non-poly-m-n-root}, this is impossible. While, under the assumption of Proposition \ref{prop: poly-m-n-not-root-case}, this in particular implies that $\alpha^n = \delta^n$ contradicting that $\alpha/\delta$ is not a root of unity. Hence, there are always infinitely many $n \in \N^+$, such that there exists a $\lambda \in \C$ satisfies 
$$ f^n(\lambda) = g^n(\lambda) = c(\lambda).$$
Now, since $\alpha$ and $\delta$ are not root of unity, we can apply the rest of the arguments in the proof verbatim in this case.
\end{rmk}
\begin{rmk}
    It is not hard to show that if $\alpha$ and $\beta$ are assumed to not be roots of unity, then Proposition \ref{prop: poly-m-n-not-root-case} and \ref{prop: non-poly-m-n-root} can be strength to an if and only if statement.

    We verify here the other direction of Proposition \ref{prop: non-poly-m-n-root} assuming $\alpha$ and $\beta$ are not roots of unity. A similar argument also works for Proposition \ref{prop: poly-m-n-not-root-case}.  
    
     Suppose $c_1(x)$ and $c_2(x)$ are given as above in the statement. Then we only need to check that there exist infinitely many $\lambda \in \C$ such that 
    $$ \mu F(\lambda)^p = \alpha^m$$
    $$ F(\lambda)^q= \delta^n$$
    for some $m,n \in \N^+$. Note that for any $(m,n) \in (\N^+)^2$ such that $\mu^q = \alpha^{mq}/\delta^{np}$, we have $\mu = \alpha^{m}/\delta'^{np}$ for some $\delta' \in \C^*$ satisfying $\delta'^q = \delta$. Then, we can find a $\lambda_{m,n} \in \C$ such that $F(\lambda_{m,n}) = \delta'^n$. Note that $\lambda_{m,n}$ satisfies both
    $$ \mu F(\lambda_{m,n})^p = \alpha^m ,$$
    $$F(\lambda_{m,n})^q = \delta^n.$$
    Then, our assumption implies that there are infinitely many such pairs of $(m,n) \in \Z^2$ and hence infinitely many distinct $\lambda_{m,n}$ satisfying the condition since there must be infinitely many distinct $n$'s in these pairs given that $\mu^q = \alpha^{mq}/\delta^{np}$ always holds for them.
\end{rmk}

\begin{prop}\label{prop: trans-mult-common}
    Suppose $f(x) = \alpha x + \beta$ and $g(x) = x +  \gamma$, where $\alpha \neq 1$ and $\alpha \in \overline{\Q}^*$. Suppose $c_1, c_2$ are two rational functions. Then there are infinitely many $\lambda \in \C$ such that 
    $$ f^n(\lambda) = c_1(\lambda)$$
    $$ g^m(\lambda) = c_2(\lambda)$$
    hold for some $m,n \in \N^+$ implies that there exists non-constant rational functions $F(x)$ and $B(x)$ and $d \in \Z$ such that 
    $$ c_1(x) =  F(x)^d\left(x - \frac{\beta}{1- \alpha}\right) + \frac{\beta}{1 - \alpha}$$
    $$ c_2(x) = \gamma B(F(x)) + x.$$
    Moreover, $B(x) \in \overline{\Q}[x] $ and either $\alpha$ or $1/\alpha$ is an algebraic integer.

    In the special case that one requires $n = m$, we have that there are only finitely many $\lambda \in \C$ such that the system of equation hold for some $n \in \N^+$.

\end{prop}
\begin{proof}
    First, we note that $\alpha$ cannot be a root of unity as otherwise there will exist a $n_0 \in \N^+$ such that 
    $f^{n_0}(\lambda) = c_1(\lambda)$
    for infinitely many $\lambda \in \C$, since $\{f^n : n \in \N^+\}$ is a finite set. This implies $f^{n_0} = c_1$ contradicting the assumption that $\{f,c_1\}$ is non-degenerate.
    
    Similarly as in the proof of Proposition \ref{prop: poly-m-n-not-root-case}, there are infinitely many $\lambda \in \C$ such that the system of equations holds for some $m,n \in \N^+$ implies that the set of $(m,n) \in (\N^+)^2$ such that
    $$ \alpha^n = \frac{c_1(\lambda) - \beta/(1 - \alpha)}{\lambda - \beta/(1 - \alpha)}$$
    $$ m = \frac{c_2(\lambda) - \lambda}{\gamma},$$
    hold for some $\lambda \in \C$ is infinite and its projection to both coordinates are infinite, since we assumed that $\{f,c_1\}$ and $\{g,c_2\}$ are non-degenerate. Let us denote this set of $(m,n) \in (\N^+)^2$ as $S$.
     This implies that the curve $C$ parametrized by the rational function
     
    $$ x \dashrightarrow \left(\frac{c_1(x) - \beta/(1- \alpha)}{x - \beta/(1- \alpha)} , \frac{c_2(x) - x}{\gamma}\right) $$
    is an algebraic curve contains $\{(\alpha^m, n): (m,n) \in S\}$. Then, by \cite{Za03} (see also the discussion following the first corollary after Theorem 1 in \cite{CZ05}), we have that there exists a non-constant rational function $B(x) \in \C(x)$ and $d \in \Z$ such that $C$ can also be parametrized by the rational map $\phi: \P^1 \dashrightarrow C$ given by
    $$  x \dashrightarrow (x^d, B(x)).$$

    Since $B(\alpha^m) = n$ for all $(m,n) \in S$ and $\alpha \in \overline{\Q}^*$, it is not hard to see that $B(x) \in \overline{\Q}(x) $. In fact, \cite[Theorem 1]{CZ05} implies that $B(x) \in \overline{\Q}[x] \cup \overline{\Q}[1/x]$. To see this, note that since $\alpha$ is not a root of unity, there exists a valuation $\nu$ such that $\nu(\alpha)>0$ or $\nu (1/\alpha) >0$. Suppose $\nu(\alpha) > 0$. Then since $B(\alpha^m) = n$ for all $(m,n) \in S$, we have $B(x)$ is a converging power series with coefficients in $\C_v$ around the origin satisfying the conditions of \cite[Theorem 1]{CZ05}. Thus, \cite[Theorem 1]{CZ05} concludes that $B(x) \in \overline{\Q}[x]$. Similarly, if $\nu(1/\alpha) > 0$, then $B(1/x)$ will satisfy the conditions of \cite[Theorem 1]{CZ05} and the argument will give that $B(1/x) \in \overline{\Q}[x]$ or equivalently $B(x) \in \overline{\Q}[1/x]$.

    We may assume that the rational map $\phi : \P^1 \dashrightarrow C$ is generically injective. Since if not, then there exists a primitive $l$-th root of unity $\mu$, where $l$ divides $d$, such that $\phi(x) = \phi(\mu x)$ for infinitely many $x \in \C$. This implies that $B(\mu x) = B(x) $, or equivalently $B(x) \in \C(x^l)$. Then we can replace the parametrization $\phi$ by 
    $$ x \dashrightarrow (x^{d/l}, B'(x))$$
    where $B'(x^l) = B(x)$.\\

    Furthermore, suppose $B(x) \in \overline{\Q}[x]$. Let $K$ be the finite field extensions of $\Q$ such that $B(x)$ and $\alpha$ are all defined over $K$. If none of $\alpha$ and $1/\alpha$ is in $\mathcal{O}_K$, then there exists a valuation $\nu$ of $K$ such that $\nu(\alpha) < 0$. Then note again that for any $(m,n) \in S$, we have 
    $$ \nu(B(\alpha^m)) = \nu(n)$$
    for all $(m,n) \in S$. However, 
    $$ \nu(B(\alpha^m)) \to -\infty $$
    as $m \to \infty$, while $\nu(n) \geq 0$, which gives the contradiction. Similarly, if $B(x) \in \overline{\Q}[1/x]$, then the assumption that neither $\alpha$ nor $1/\alpha$ is in $\mathcal{O}_K$ will imply the existence of a valuation $\nu$ such that $\nu (1/\alpha) < 0$. Then the same argument will give the contradiction. Thus, we can conclude either $\alpha$ or $1/\alpha$ is an algebraic integer.\\

    Now, since $\psi : \P^1 \dashrightarrow C$ given by 
    $$ x \dashrightarrow \left(\frac{c_1(x) - \beta/(1- \alpha)}{x -\beta/(1 - \alpha)} , \frac{c_2(x) - x}{\gamma}\right)$$
    is another parametrization of the curve $C$, we have 
    $$ F \coloneqq \phi^{-1} \circ \psi : \P^1 \dashrightarrow \P^1 $$
    is a non-constant rational function satisfying that 
    $$ \psi = \phi \circ F.$$
    Explicitly, this implies that 
    $$ \frac{c_1(x) - \beta/(1- \alpha)}{x -\beta/(1 - \alpha)} = F(x)^d,$$
    $$\frac{c_2(x) - x}{\gamma} = B(F(x)). $$
    
    Hence 
    $$c_1(x) =  F(x)^d\left(x - \frac{\beta}{1 - \alpha}\right) + \frac{\beta}{1 - \alpha}$$
    $$c_2(x) = \gamma B(F(x)) + x.$$

    We remark that we can always adjust the sign of $d \in \Z$ and replace $F(x)$ by $1/F(x)$ so that $B(x) \in \overline{\Q}[x]$.\\

    Now, we look at the case when $n = m$. Then, there are infinitely many $\lambda \in \C$ such that the equations hold for some $n \in \N^+$ implies that $C$ is an irreducible algebraic subvariety intersecting $\Orb_{m_\alpha, t_1}((1,0))$ infinitely many times, where $m_\alpha(x) = \alpha x$ and $t_1(y) = y + 1$
    for $x\in \G_m$ and $y \in \G_a$. Then, by the Skolem-Mahler-Lech theorem \cite{Le53}, we have 
    $$ C = (\alpha^i, i) + \overline{\{ (\alpha^{nk}, nk) : k \in \N^+, \forall n \in \N\}},$$
    where $i \in \N$. Note that 
    $H \coloneqq \overline{\{ (\alpha^{nk}, nk) : k \in \N^+, \forall n \in \N\}}$
    is an algebraic subgroup generated by $(\alpha^k, k)$ and the above implies $\dim(H) = \dim(C) = 1$. Let $\pi_1, \pi_2$ be the two projections from $H$ to $\G_m$ and $\G_a$ respectively. 
    
    Suppose both of them are dominant, then $\pi_1$ is a finite morphism with kernel in $$\{1\} \times \G_a.$$ However, the only finite subgroup in it is the trivial group. Thus, $\pi_1$ is actually an isomorphism. Therefore, $\pi_2 \circ \pi^{-1}_1 : \G_m \to \G_a$ gives an non-trivial homomorphism between $\G_m$ and $\G_a$. This is impossible as there is no non-trivial homomorphism between $\G_m$ and $\G_a$. 

    Thus, $C$ must be either vertical or horizontal, which contradicts the assumption that it contains infinitely many points in $\{(\alpha^n, n): n \in \N^+\}$.

\end{proof}

Note that we can also conclude that if $\alpha \in \C^*$ is transcendental in above, then there are only finitely many $\lambda \in \C$ making the condition hold for any rational functions $c_1$ and $c_2$.

\begin{prop}\label{prop: add-mul-trance-case}
Suppose \( f(x) = \alpha x + \beta \) and \( g(x) = x + \gamma \), where \( \alpha \in \C^* \) is a transcendental number. Let \( c_1, c_2 \) be rational functions. Then there exist only finitely many \( \lambda \in \C \) such that 
\[
f^m(\lambda) = c_1(\lambda) \quad \text{and} \quad g^n(\lambda) = c_2(\lambda)
\]
for some \( m, n \in \N^+ \).
\end{prop}

\begin{proof}

Let \(K\) be the finitely generated field obtained by adjoining to \(\overline{\mathbb{Q}}\) all coefficients of \(f,g,c_1,c_2\). By hypothesis \(\operatorname{trdeg}(K/\overline{\mathbb{Q}})\ge 1\).

Assume for contradiction that there are infinitely many \(\lambda\in\mathbb{C}\) for which
\[
f^m(\lambda)=c_1(\lambda)\quad\text{and}\quad g^n(\lambda)=c_2(\lambda)
\]
for some \(m,n\in\mathbb{N}^+\). Arguing as in the proof of Proposition~\ref{prop: trans-mult-common}, we obtain an affine curve
\[
C \;=\; V(P(x,y))\subset\mathbb{A}^2,
\]
defined over \(\overline{K}\), which admits the rational parametrization
\[
t \dashrightarrow\Big(\frac{c_1(t)-\beta/(1-\alpha)}{t-\beta/(1-\alpha)},\;\frac{c_2(t)-t}{\gamma}\Big),
\]
and which contains infinitely many points of the form \((\alpha^m,n)\) with \(m,n\in\mathbb{N}^+\).

Choose a subfield \(L\subseteq K\) with \(\operatorname{trdeg}(L/\overline{\mathbb{Q}})=1\) and a specialization
\(\sigma:K\to L \) such that \(\alpha_0\coloneqq\sigma(\alpha)\) is transcendental over \(\overline{\mathbb{Q}}\). Put \(P'(x,y)=\sigma(P)\). By construction, \(P'(\alpha_0^m,n)=0\) for infinitely many pairs \((m,n)\in(\mathbb{N}^+)^2\); denote by \(S\subset(\mathbb{N}^+)^2\) this infinite set of solutions.

Because the pairs \(\{f,c_1\}\) and \(\{g,c_2\}\) are non-degenerate, we can extract an infinite subsequence
\(I=\{(m_i,n_i)\}_{i\ge 1}\subset S\) with \(m_1<m_2<\cdots\) and with \(\{n_i\}_{i \geq 1}\) infinite.

Write the leading \(x\)-term of \(P'(x,y)\) as \(a_d(y)x^d\). Choose \(i\) large so that \(m_i\) exceeds twice of the absolute value of the maximal exponent that appears in any Puiseux expansion of the coefficients of \(P'\) viewed as series in \(\alpha_0\); moreover choose \(i\) even larger if necessary with \(a_d(n_i)\neq0\). Then 
\[
\deg_{\alpha_0}\big(a_d(n_i)\alpha_0^{m_i d}\big)
>\deg_{\alpha_0}\big(P'(\alpha_0^{m_i},n_i)-a_d(n_i)\alpha_0^{m_i d}\big),
\]
which contradicts \(P'(\alpha_0^{m_i},n_i)=0\). Hence there are only finitely many such \(\lambda\).

\end{proof}

\begin{prop}\label{prop: rational-trans-multi-nocommon}
    Suppose $f(x) = x +  \beta$ and $g(x) = x/(\gamma x + \delta)$ where $\delta \in \overline{\Q}^*$ is not $1$ and $\beta, \gamma \in \C^*$. Let $c_1$ and $c_2$ be two rational functions. Then there are infinitely many $\lambda \in \C$ such that 
    $$ f^n(\lambda) = c_1(\lambda)$$
    $$ g^m(\lambda) = c_2(\lambda)$$
    hold for some $m,n \in \N^+$ implies that there exist non-constant rational functions $F(x)$ and $B(x)$ and $d \in \Z$ such that 
    $$ c_1(x) = \beta B(F(x)) + x$$
    $$ c_2(x) = x/\left((1 - F(x)^d)\frac{\gamma}{1 - \delta}x + F(x)^d\right).$$
    Moreover, $B(x) \in \overline{\Q}[x] $ and either $\delta$ or $1/\delta$ is an algebraic integer.

     In the special case that one requires $n = m$, we have that there are only finitely many $\lambda \in \C$ such that the system of equations holds for some $n \in \N^+$.

\end{prop}
\begin{proof}
    Note that $\delta$ cannot be a root of unity given our assumption that $\{g,c_2\}$ is non-degenerate. Since otherwise, infinitely many $\lambda \in \C$ satisfying the equations will result in that $g^m \equiv c_2$ for a $m \in \N^+$ as $\{g^m : m \in \N^+\}$ is a finite set.\\
    
    Suppose there are infinitely many $\lambda \in \C$ such that 
    $$ f^n(\lambda) = c_1(\lambda)$$
    $$ g^m(\lambda) = c_2(\lambda)$$
    hold for some $n,m \in \N^+$. Note that $f^n(x) = x + n \beta$ and $$ g^m (x) = \frac{x}{x(1 - \delta^m)\gamma/( 1- \delta) + \delta^m}.$$
    Then we have that there are infinitely many $\lambda \in \C$ such that 
    $$ n = (c_1(\lambda) - \lambda) / \beta,$$
    $$ \delta^m = \frac{\lambda}{1 - \gamma\lambda/(1 - \delta)} \left(\frac{1}{c_2(\lambda)} - \frac{\gamma}{1 - \delta}\right)$$
    for some $n,m \in \N^+$. This implies that the curve $C \subseteq \G_a \times \G_m$ parametrized by the rational map
    $$  x \dashrightarrow \left((c_1(x) - x) / \beta, \frac{x}{1 - x\gamma/(1 - \delta)} \left(\frac{1}{c_2(x)} - \frac{\gamma}{1 - \delta}\right)\right)  $$
    contains infinitely many points in $$\{(n,\delta^m) : n,m \in \N^+\}.$$

    Then, by the exactly same argument in Proposition \ref{prop: trans-mult-common}, we have that there exists a non-constant rational function $F$, a positive integer $d$ and a rational function $B(x) \in \overline{\Q}[x]\cup \overline{\Q}[1/x]$ such that 
    $$(c_1(x) - x) / \beta = B(F(x)) ,$$
    $$ \frac{x}{1 - x\gamma/(1 - \delta)} \left(\frac{1}{c_2(x)} - \frac{\gamma}{1 - \delta}\right) = F(x)^d.$$
    Moreover, either $\delta$ or $1/\delta$ is an algebraic integer. Also, we can again adjust the sign of $d \in \Z$ and replace $F(x)$ with $1/F(x)$ to assume that $B(x) \in \overline{\Q}[x]$.
    
    Then the conclusion follows from a straightforward manipulation of this system of equations.

    In the special case that $n =m $, again by the Skolem-Mahler-Lech theorem \cite{Le53}, we have that $C$ must be a translation of an algebraic subgroup of $\G_a \times \G_m$. Hence, the same argument in Proposition \ref{prop: trans-mult-common} implies that $C$ must be either vertical or horizontal. This contradicts that $C$ contains infinitely many points in $\{(n, \delta^n) : n \in \N^+\}$.
    
\end{proof}
\begin{prop}\label{prop: rational-add-mul-trance-case}
Suppose \( f(x) =  x + \beta \) and \( g(x) =x/( \gamma x + \delta) \), where \( \delta \in \C^* \) is a transcendental number. Let \( c_1, c_2 \) be rational functions. Then there exist only finitely many \( \lambda \in \C \) such that 
\[
f^n(\lambda) = c_1(\lambda) \quad \text{and} \quad g^m(\lambda) = c_2(\lambda)
\]
for some \( m, n \in \N^+ \).
\end{prop}
\begin{proof}

    The proof of Proposition~\ref{prop: add-mul-trance-case} essentially shows that an affine curve that is neither vertical nor horizontal cannot contain an infinite sequence of points of the form  
\[
\{(n, \delta^m) : (n, m) \in I\},
\]
where \( I = \{(m_i, n_i) : i \in \N\} \subseteq (\N^+)^2 \) with \( m_1 < m_2 < \cdots \) and \( \{n_i : i \in \N\} \) infinite.  

However, suppose there exist infinitely many such \( \lambda \in \C \). Since the pairs \( \{f, c_1\} \) and \( \{g, c_2\} \) are non-degenerate, the argument in the proof of Proposition~\ref{prop: rational-trans-multi-nocommon} implies that the curve \( C \subseteq \A^2 \) parametrized by  
\[
x \dashrightarrow \left( \frac{c_1(x) - x}{\beta}, \, \frac{x}{1 - x\gamma/(1 - \delta)} \left(\frac{1}{c_2(x)} - \frac{\gamma}{1 - \delta}\right) \right)
\]
must contain such an infinite sequence of points \( (n, \delta^m) \).  
Since \( C \) is clearly neither vertical nor horizontal, this leads to a contradiction.

\end{proof}

\begin{rmk}\label{rmk: n=m-nondeg-2}
    Similarly, the finiteness statement of Proposition \ref{prop: trans-mult-common}, \ref{prop: add-mul-trance-case}, \ref{prop: rational-trans-multi-nocommon} and \ref{prop: rational-add-mul-trance-case} remain valid without assuming $\{f,c_1\}$ and $\{g,c_2\}$ are non-degenerate providing that $n = m \in \N^+$, $c_1 = c_2$, $\gamma \neq 0$ and $\alpha/\delta$ is not a root of unity. We only need to handle the case that there exists a $n \in \N^+$ and an infinite set of $\lambda \in \C$ such that 
    $$ f^n(\lambda) = c(\lambda)$$
    $$ g^n(\lambda) = c(\lambda)$$
    where $c \coloneqq c_1 = c_2$. This is impossible under the assumption of Propositions \ref{prop: rational-trans-multi-nocommon} and \ref{prop: rational-add-mul-trance-case}. While, under the assumption of Propositions \ref{prop: trans-mult-common} and \ref{prop: add-mul-trance-case}. Since the assumption of these propositions is that one of $\alpha$ and $\delta$ is $1$, we have these imply in particular that $\alpha^n = \delta^n = 1 $ and hence $\alpha/\delta$ is a root of unity. The rest of the cases where for each $n \in \N^+$ there are at most finitely many $\lambda$ makes the above equations hold follow verbatim from the proof of these propositions.
\end{rmk}

The propositions proved in this subsection resolve Question~\ref{qu: automorphisms-common-zeros} 
except in the case where both \( f \) and \( g \) are translation maps. 
In this situation, the problem essentially reduces to the classification of curves containing 
infinitely many integer points \cite{ABP09}, which contains a very broad class of curves. 
Consequently, it is unrealistic to expect a classification result as explicit as those obtained 
in the preceding propositions.

Nevertheless, for the purposes of many applications---and in particular for those considered in 
this paper, where \( f \) and \( g \) are assumed to generate a free semigroup under composition---this exceptional case need not be considered. Indeed, when both \( f \) and \( g \) are translations, they necessarily commute and, therefore, cannot generate a free semigroup.

We summarize the classification we obtained in this subsections in the following two Theorems: 
\begin{thm}[Theorem \ref{thm: main-classification-common-root}]
    Let $f(x) = \alpha x + \beta$ and $g(x) = \delta x + \gamma$, where $\alpha, \delta \in \C^*$ and $\beta, \gamma \in \C$. Let $c_1$ and $c_2$ be two rational functions such that $\{f, c_1\}$ and $\{g,c_2\}$ are non-degenerate. Then there are infinitely many $\lambda \in \C$ such that 
    $$ f^m(\lambda) = c_1(\lambda)$$
    $$ g^n(\lambda) = c_2(\lambda)$$
    holds for some $m,n \in \N^+$ implies the following holds:
    \begin{enumerate}
        \item If $\alpha$ and $\delta$ are not $1$, then 
         $$ c_1(x) = \mu F(x)^p\left (x - \frac{\beta}{1-\alpha}\right) + \frac{\beta}{1 - \alpha}$$
    $$c_2(x) = F(x)^q\left(x - \frac{\gamma}{1 - \delta}\right) + \frac{\gamma}{1 - \delta},$$
    where $F(x)$ is a non-constant rational function, $p,q$ are non-zero coprime integers and $\mu \in \C^*$ such that $\mu^q= \alpha^{mq}/\delta^{np}$ for infinitely many pairs of $(m,n) \in (\N^+)^2$.
    \item If $\alpha \in \C^* \setminus \{1\}$ and $\delta = 1$, then $\alpha \in \overline{\Q}^*$ and there exist non-constant rational functions $F(x)$ and $B(x)$ and $d \in \Z$ such that 
    $$ c_1(x) =  F(x)^d\left(x - \frac{\beta}{1- \alpha}\right) + \frac{\beta}{1 - \alpha}$$
    $$ c_2(x) = \gamma B(F(x)) + x.$$
    Moreover, $B(x) \in \overline{\Q}[x] $ and either $\alpha$ or $1/\alpha$ is an algebraic integer.
    \end{enumerate}
\end{thm}
\begin{proof}
    This is a combination of Propositions \ref{prop: poly-m-n-not-root-case}, \ref{prop: trans-mult-common} and \ref{prop: add-mul-trance-case}. Note that Proposition \ref{prop: add-mul-trance-case} rules out that $\alpha \in \C^* \setminus \overline{\Q}^*$ in the second case.
\end{proof}
\begin{thm}[Theorem \ref{thm: main-classification-no-common-root}]
    Let $f(x) = \alpha x + \beta$ and $g(x) = x/(\gamma x + \delta)$, where $\alpha, \delta, \beta, \gamma \in \C^*$. Let $c_1$ and $c_2$ be two rational functions such that $\{f, c_1\}$ and $\{g,c_2\}$ are non-degenerate. Then there are infinitely many $\lambda \in \C$ such that 
    $$ f^m(\lambda) = c_1(\lambda)$$
    $$ g^n(\lambda) = c_2(\lambda)$$
    holds for some $m,n \in \N^+$ implies the following holds:
    \begin{enumerate}
        \item If $\alpha$ and $\delta$ are not $1$, then $$ c_1(x) = \left(x- \frac{\beta}{1 - \alpha}\right)\mu F(x)^p + \frac{\beta}{1 - \alpha}$$
        $$ c_2(x) = \frac{x}{F(x)^q (1 - \gamma x/(1- \delta)) + \gamma /(1 - \delta)},$$
         where $F(x)$ is some non-constant rational function, $p,q$ are some coprime non-zero integers and $\mu \in \C^*$ such that $\mu^q= \alpha^{mq}/\delta^{np}$ for infinitely many pairs of $n,m \in \N^+$.
         \item If $\delta \in \C^* \setminus \{1\}$ and $\alpha = 1$, then $\delta \in \overline{\Q}^*$ and there exist non-constant rational functions $F(x)$ and $B(x)$ and $d \in \Z$ such that 
    $$ c_1(x) = \beta B(F(x)) + x$$
    $$ c_2(x) = x/\left((1 - F(x)^d)\frac{\gamma}{1 - \delta}x + F(x)^d\right).$$
    Moreover, $B(x) \in \overline{\Q}[x] $ and either $\delta$ or $1/\delta$ is an algebraic integer.
    \end{enumerate}
\end{thm}
\begin{proof}
    This is a combination of Propositions \ref{prop: non-poly-m-n-root}, \ref{prop: rational-trans-multi-nocommon} and \ref{prop: rational-add-mul-trance-case}. Note that Proposition \ref{prop: rational-add-mul-trance-case} rules out that $\delta \in \C^* \setminus \overline{\Q}^*$ in the second case.
\end{proof}

\subsection{Applications on the system of $\{f,g,c\}$}
In this subsection, we demonstrate a few applications of the classification results obtained above. In particular, we strengthen \cite[Theorem 1]{HT17} (Theorem \ref{thm: f^m-g^n-c-poly}) and also provide an alternative proof to \cite[Theorem 1.3]{NZ25}.

\begin{cor}\label{cor: mul-algebraic-case}
    Suppose $f(x) = \alpha x + \beta$ and $g(x) = \delta x + \gamma$ and $\alpha, \delta \in \C^*$ and they are not $1$. Assume $c$ is a polynomial and $\{f,c\}$, $\{g,c\}$ are non-degenerate. Then, there are infinitely many $\lambda \in \C$ such that
    $$ f^m(\lambda) = g^n(\lambda) = c(\lambda)$$
    holds for some $(m,n) \in (\N^+)^2$ implies that $f$ and $g$ won't generate a free semigroup under compositions.
    
\end{cor}
\begin{proof}
    By Proposition \ref{prop: poly-m-n-not-root-case}, $c(x)$ must satisfy 
    \begin{equation}\label{eq: poly-m-n-free-1}
        c(x) = \mu F(x)^p \left(x - \frac{\beta}{1 - \alpha}\right) + \frac{\beta}{ 1- \alpha} = F(x)^q\left(x - \frac{\gamma}{1 - \delta }\right) + \frac{\gamma}{1 - \delta}
    \end{equation}
    for some $\mu \in \C$, non-constant rational function $F(x)$ and a pair of coprime non-zero integers $p,q$. Moreover, $\mu^q = \alpha^{mq}/\delta^{np}$ for infinitely many distinct pair of $(m,n) \in (\N^+)^2$.  \\

    We first handle the case when $p$ and $q$ are different signs. Without loss of generality, we assume that $p$ is positive and $q$ is negative. Then, denoting $F(x) = P(x)/Q(x)$ for a pair of coprime polynomials $P$ and $Q$, we have
    \begin{equation}\label{eq: p>0-q<0-f-g-c-p-1}
        c(x) = \mu \left(\frac{P(x)}{Q(x)} \right)^p \left(x - \frac{\beta}{1 - \alpha}\right) + \frac{\beta}{ 1- \alpha} = \left(\frac{Q(x)}{P(x)} \right)^{-q}\left(x - \frac{\gamma}{1 - \delta }\right) + \frac{\gamma}{1 - \delta}.
    \end{equation}
    Since $c(x)$ is a polynomial, we must have $$Q(x)^{-p}\left(x - \frac{\beta}{1 - \alpha}\right)$$
    $$ P(x)^{q}\left(x - \frac{\gamma}{1 - \delta }\right)$$
    are polynomials.

 Let us first suppose $Q(x)$ is a constant, then, since $F(x)$ is non-constant, we have 
    $$ F(x) = k\left(x - \frac{\beta}{1 - \alpha}\right)$$
    for some $k \in \C^*$ and $q = 1$. Then Equation (\ref{eq: p>0-q<0-f-g-c-p-1}) implies that
    $$ c(x) = \mu k^p \left(x - \frac{\beta}{1 - \alpha}\right)^{p+1} + \frac{\beta}{1- \alpha} = k^{-1} + \frac{\gamma}{1 - \delta},$$
    which is a contradiction.

    Now, suppose $Q(x)$ is not a constant, then the first equality of Equation (\ref{eq: p>0-q<0-f-g-c-p-1}) implies that $p = 1$ and $$F(x) = P(x)/\left(x - \frac{\beta}{1 - \alpha}\right).$$ If $P(x)$ is a constant, then the same argument as above gives us a contradiction. Thus, together with the second equality of Equation (\ref{eq: p>0-q<0-f-g-c-p-1}), we have that 
    $$ F(x) = k\left(x - \frac{\gamma}{1 - \delta }\right)/\left(x - \frac{\beta}{1 - \alpha}\right)$$
    for some $k \in \C^*$ and $q = -1$. Now, Equation (\ref{eq: p>0-q<0-f-g-c-p-1}) becomes 
    \begin{equation}
        c(x) = \mu k\left(x - \frac{\gamma}{1 - \delta }\right) + \frac{\beta}{1 - \alpha} = k^{-1}\left(x - \frac{\beta}{1 - \alpha}\right) + \frac{\gamma}{ 1- \delta}.
    \end{equation}
    Thus, 
    $$ \mu k  =k^{-1} $$
    $$ (k^{-1} + 1) \frac{\beta}{1- \alpha} = (k^{-1} + 1)\frac{\gamma}{1- \delta}.$$
    Note that if $ k^{-1} + 1 \neq 0$ then 
    $$ \frac{\beta}{1- \alpha} = \frac{\gamma}{1- \delta}$$
    implies that $f$ and $g$ are conjugated together to two scaling maps and hence commute to each other. 
    
    If $ k^{-1} = -1$, then $\mu = 1$. Then, since we can find two distinct pairs of $(m_1,n_1), (m_2, n_2) \in (\N^+)^2$ such that $$1 = \mu^q = \alpha^{m_1q} \delta^{-n_1p} = \alpha^{m_2q}\delta^{-n_2p},$$
    we have 
    $$ f^{-m_1q} \circ g^{n_1p} (x) =  x + \nu_1 $$
    $$ f^{-m_2q} \circ g^{n_2p} (x) =  x + \nu_2$$
    for some $\nu_1, \nu_2 \in \C$. Then, they are commuting to each other and hence $f$ and $g$ won't generate a free semigroup under composition.\\

    Now, we suppose that $p$ and $q$ are of equal sign and without loss of generality we assume that they are both positive. Since we require that $c(x)$ is a polynomial, if $F(x)$ is not a polynomial then $F(x) = P(x)/Q(x)$, for some coprime polynomials $P,Q$, where $$Q(x)^{-p}(x - \beta/(1 - \alpha)),$$
    $$ Q(x)^{-q}(x - \gamma/(1 - \delta))$$
    are polynomials. This is only possible if $Q$ is a constant or $$\beta/(1 - \alpha) =\gamma/(1 - \delta).$$ The later implies that $f$ and $g$ share two common fixed points and, therefore, commute to each other, as they can be simultaneously conjugated to scaling maps. 

    Now, we assume that $Q(x)$ is a constant, then $F(x)$ is a polynomial. Then, 
    Equation (\ref{eq: poly-m-n-free-1}) implies that $\deg(F^p) = \deg(F^q)$. Since $F(x)$ is not a constant, we have $p = q$. Then by looking at the leading coefficients of the highest degree term in Equation (\ref{eq: poly-m-n-free-1}), we have $\mu = 1$ and, denoting $H(x) \coloneqq F(x)^p = F(x)^q$, 
    $$ H(x) (x - \beta/(1 - \alpha)) + \beta/(1 - \alpha) = H(x)(x - \gamma/(1 - \delta)) + \gamma/(1 - \delta).$$
    This is equivalent to 
    $$ -\beta/(1- \alpha)H(x) + \beta/(1 - \alpha) = -\gamma/(1 - \delta)H(x) + \gamma/(1 - \delta).$$
    Now, looking at the coefficients of $\deg(H(x))$ terms on both sides of the equation, we have 
    $$ \beta/(1 - \alpha) = \gamma/(1 - \delta).$$
    This again implies that $f$ and $g$ commute to each other.
    
\end{proof}
\begin{prop}\label{prop: non-trance-add-mul-nonfree}
    Let $f(x) = \alpha x + \beta $ and $g(x) = x + \gamma$. Suppose $\alpha \in \overline{\Q}^*$. Then the semigroup generated by $f(x)$ and $g(x)$ under compositions are not free.
\end{prop}
\begin{proof}
    Since $\alpha \in \overline{\Q}^*$, there exists a minimal polynomial $h(x) \in \Z[x]$ such that $h(\alpha) = 0$. Let us group the terms of $h(x)$ according to the sign of its coefficients and we get
    $$ h(x) = \sum^d_{i = 0} a_ix^i - \sum^d_{i = 0}b_i x^i$$
    where $a_i, b_i \in \N$.
    Then let us consider two composition sequences 
    $$ g^{a_0} \circ f \circ g^{a_1} \circ f \circ g^{a_2} \circ \dots \circ f \circ g^{a_d} = \alpha^d x + \frac{1 - \alpha^d}{1 - \alpha} \beta + \sum^d_{i = 0}a_i\alpha^i \gamma$$
    and 
    $$ g^{b_0} \circ f \circ g^{b_1} \circ f \circ g^{b_2} \circ  \dots \circ f \circ g^{b_d} = \alpha^d x + \frac{1 - \alpha^d}{1 - \alpha}
    \beta + \sum^d_{ i =0} b_i \alpha^i \gamma.$$

    Since $h(\alpha) = 0$, we have 
    $$g^{a_0} \circ f \circ g^{a_1} \circ f \circ g^{a_2} \circ \dots \circ f \circ g^{a_d} = g^{b_0} \circ f \circ g^{b_1} \circ f \circ g^{b_2} \circ  \dots \circ f \circ g^{b_d},  $$
    which implies that the semigroup generated by $f$ and $g$ are not free since we assumed that $\{a_0, \dots, a_d\}$ and $\{b_0, \dots, b_d\}$ are distinct sequences of non-negative integers.
\end{proof}

\begin{thm}\label{thm: f-g-m-n-linear-finite}
    Let $f$, $g$ and $c$ be polynomials over $\C$ and $f = \alpha x + \beta$ and $g = \delta x + \gamma$ are automorphisms with $\alpha, \beta \in \C^*$. Suppose $c(x)$ is not an iterate of $f$ or $g$. Suppose $f$ and $g$ generate a free semigroup under composition. Then there are only finitely many $\lambda \in \C$ such that  
    $$ f^m(\lambda) = g^n(\lambda) = c(\lambda)$$
    for some $m,n \in \N^+$.
\end{thm}
\begin{proof}
    We first suppose that both $\alpha$ and $\delta$ are not $1$. On the contrary that, if there are infinitely many $\lambda \in \C$ such that 
    $$ f^m(x) = g^n(x) = c(x)$$ for some $m,n \in \N^+$,
    
    then by Corollary \ref{cor: mul-algebraic-case}, we have that $f, g$ won't generate a free semigroup under composition, which is a contradiction.

    Now, let's suppose one of $\delta$ and $\alpha$ is $1$. Without loss of generality, we assume that it is $\delta$. If $\alpha \notin \overline{\Q}$, we have by Proposition \ref{prop: add-mul-trance-case} that there are only finitely many $\lambda \in \C$ such that $f^n(\lambda) = g^m(\lambda) = c(\lambda)$ for some $m,n \in \N^+$. Suppose $\alpha \in \overline{\Q}$, by Proposition \ref{prop: non-trance-add-mul-nonfree}, we know that $f$ and $g$ won't generate a free semigroup under composition. 

    Lastly, if both $\alpha, \delta = 1$, then obviously that $f$ and $g$ commute to each other and don't generate a free semigroup. 
    
\end{proof}

With Theorem \ref{thm: f-g-m-n-linear-finite}, we can quickly obtain the proof of Theorem \ref{thm: f^m-g^n-c-poly}.
\begin{proof}[Proof of Theorem \ref{thm: f^m-g^n-c-poly}]
    The case when at least one of $f$ and $g$ is of degree greater than $1$ is handled by \cite[Theorem 1]{HT17}. The rest of the case that $f$ and $g$ are both automorphisms is proved by Theorem \ref{thm: f-g-m-n-linear-finite}. Thus, the theorem is proved.
\end{proof}

Below, using our new approach, we present an alternative proof of \cite[Theorem 1.3]{NZ25} (Theorem \ref{thm: mainaut}). For the convenience of the reader, we repeat the statement here.

\begin{thm}
    Let $f(x)$, $g(x)$ be automorphisms on $\P^1$ defined over $\C$ and $c(x)$ be a rational function defined over $\C$. If the semigroup generated by $f(x)$ and $g(x)$ under compositions is free, then there are only finitely many $\lambda \in \C$ such that 
    \begin{equation} \label{eq: targeteq1}
        c(\lambda) = f^n(\lambda) = g^n(\lambda)
    \end{equation}
    for some positive integer $n$ unless $f(x)$, $g(x)$ are simultaneously conjugated by an automorphism on $\P^1(\C)$ to one of the following:
    \begin{enumerate}
        \item $\alpha x + \beta$, $x/(\gamma x + \delta)$, with some $\alpha, \delta, \gamma, \beta \in \C^*$ such that one of $\alpha /\delta $ and $\alpha \delta$ is a root of unity;
        \item $\alpha x + \beta$, $\delta x + \gamma$,  with some $\alpha, \delta \in \C^*$ such that $\alpha$ and $\delta$ are not roots of unity, $\gamma$ and $\beta$ are not both $0$, $\alpha/\delta$ is a root of unity other than $1$ or one of $\alpha^2/\delta$ and $\alpha/
        \delta^2$ is a root of unity.
    \end{enumerate}
    
\end{thm}
\begin{proof}
    Note that if $f$ and $g$ share more than $1$ common fixed points, then they are commuting to each other. So we assume that they are sharing at most one common fixed point. We will show that if we assume the two exceptional conditions doesn't hold for our $f$ and $g$, then there are only finitely many $\lambda$ making Equation (\ref{eq: targeteq1}) hold for some $n \in \N^+$.

    {\bf Case \RNum{1}:} Suppose first that they share exactly one common fixed point. Then, our assumption implies that after a suitable conjugation by an automorphism on $\P^1$, we have 
    $$ f(x) = \alpha x + \beta$$
    $$ g(x) = \delta x + \gamma$$
    where $\alpha, \delta \in \C^*$ and $\gamma, \beta \in \C$.\\

    Although we did not assume that the pairs \( \{f,c\} \) and \( \{g,c\} \) are non-degenerate here, 
Remarks~\ref{rmk: n=m-no-degen} and~\ref{rmk: n=m-nondeg-2} indicate that the propositions in this section 
remain applicable provided that the two exceptional cases (1) and (2) do not occur, except possibly in the following three cases:
\begin{itemize}
    \item $\alpha = 1$ and $\beta = 0$ or $\delta = 1$ and $\gamma = 0$;
    \item one of $\alpha$ and $\delta$ is a root of unity other than $1$;
    \item $\alpha = \delta $.
\end{itemize}
Firstly, if $\alpha = 1$ and $\beta = 0$, then $f(x) = x$ which contradicts that $f$ and $g$ generate a free semigroup under compositions. The same argument also rules out the case of $\delta = 1$ and $\gamma = 0$. While, when one of $\alpha$ and $\delta$ is a root of unity other than $1$, without loss of generality we assume that it is $\alpha$. Then there exists a $m \in \N^+$ so that $f^m(x) = x$, which contradicts that $f$ and $g$ generate a free semigroup under compositions.  
Thus, Propositions~\ref{prop: poly-m-n-not-root-case} and~\ref{prop: add-mul-trance-case} apply in the present setting. Now, when \( \alpha = \delta  \) and none of them are roots of unity, if there are infinitely many $\lambda \in \C$ such that 
 $$ \alpha^n\lambda + \frac{1- \alpha^n}{1- \alpha}\beta = f^n(\lambda) = g^n(\lambda) = \alpha^n \lambda + \frac{1 - \alpha^n}{1 - \alpha}\gamma$$
 for some $n \in \N^+$, then we have
 $\beta = \gamma$ and hence $f =g$ contradicting the assumption that $f$ and $g$ generate a free semigroup. In the case that $\alpha = \delta = 1$, we have obviously that $f \circ g= g \circ f$ contradicts that $f$ and $g$ generate a free semigroup under composition.
\\

    Suppose at least one of $\alpha$ and $\delta$ is $1$. First of all, if $\alpha = \delta = 1$, as we discussed above, then $f$ and $g$ commute, which is also a contradiction. Suppose exactly one of them is $1$, without loss of generality, say it is $\delta =1$, and $\gamma \neq 0$ (the case $\gamma  =0$ is discussed above), then Proposition \ref{prop: add-mul-trance-case} and Proposition \ref{prop: non-trance-add-mul-nonfree} imply that either $f$ and $g$ won't generate a free semigroup or there are only finitely many $\lambda$ making Equation (\ref{eq: targeteq1}) hold for some $n \in \N^+$.

    Now, we suppose that $\alpha, \delta $ are not roots of unity. Suppose there are infinitely many $\lambda$ such that Equation \ref{eq: targeteq1} hold for some $n \in \N^+$. Then Proposition \ref{prop: poly-m-n-not-root-case} implies that there exist a root of unity $\mu$ and a pair of coprime non-zero integers $p$, $q$ such that 
    $$ \mu^q = \alpha^{nq}/\delta^{np}$$ for infinitely many $n \in \N^+$ and a pair of coprime polynomials $A(x), B(x) \in \C[x]$ such that 
    \begin{equation} \label{eq: poly-n-noroot-1}
        \mu (A(x)/B(x))^p (x - \beta/(1 - \alpha)) + \beta/(1 - \alpha) = (A(x)/B(x))^q(x - \gamma/(1 - \delta)) + \gamma/(1- \delta).
    \end{equation} 

    Without loss of generality, we can always assume that $|p| \geq  |q|$. Suppose that exactly one of $p, q$
 is negative, without loss of generality, let us assume $p < 0$. Note that by interchanging the roles of $f$  and $g$ and replacing $(p,q)$ by $(-p, -q)$, we can always make $|p|\geq |q|$ and $p < 0$ hold. Then Equation \ref{eq: poly-n-noroot-1} implies
 \begin{equation}
     \mu B(x)^{-p }A(x)^{p-q}\left(x - \frac{\beta}{1 -\alpha}\right) + \frac{\beta}{1 - \alpha} A(x)^{-q} = B(x)^{- q}\left(x - \frac{\gamma}{1 - \delta}\right) + \frac{\gamma}{ 1- \delta}A(x)^{-q}
 \end{equation}
 hold. However, since $A(x)$ and $B(x)$ are coprime in $\C[x]$, there exists a linear factor $x - \nu$ dividing only one of $A(x)$ and $B(x)$. Suppose it divides $A(x)$, then if $p < -1$ then we have
\begin{align*}
    \val_{x - \nu}\left(\mu B(x)^{-p }A(x)^{p-q} \left(x - \frac{\beta}{1 -\alpha}\right) + \frac{\beta}{1 - \alpha} A(x)^{-q}\right) \\< \val_{x-\nu} \left(B(x)^{- q}\left(x - \frac{\gamma}{1 - \delta}\right) + \frac{\gamma}{ 1- \delta}A(x)^{-q}\right)
\end{align*} 
 since $-p + q > \max\{q, -p\}$.
 This contradicts Equation \ref{eq: poly-n-noroot-1}. If $p = -1$, then $q = 1$ and we have $\mu = \alpha^n\delta^n$ for infinitely many $n \in \N^+$, which implies that $\alpha \delta$ is a root of unity. Then $f \circ g$ and $g \circ f$ commute to each other, contradicting the assumption that $f$ and $g$ generate a free semigroup under composition.

 On the other hand, if there does not exist such a linear factor $x- \nu $ dividing $A(x)$, then $A(x)$ is a constant and $B(x)$ is not. We then have Equation (\ref{eq: poly-n-noroot-1}) becomes
 \begin{equation}
     \mu B(x)^{-p}A^{p-q} \left(x - \frac{\beta}{1- \alpha}\right) + \frac{\beta}{1- \alpha} A^{-q} = B(x)^{-q}\left(x - \frac{\gamma}{1- \delta}\right) + \frac{\gamma}{1- \delta}A^{-q}. 
 \end{equation}
 Then the degree of left hand side is strictly larger than the degree of the right hand side, which is a contradiction.

 Now, we suppose both $p$ and $q$ are positive.  When $p \leq 2$, we have 
 $$ \alpha^q/\delta^p = \mu^q$$
 is a root of unity, which satisfies the exceptional condition (2) except when $p =q =1$ and $\mu = 1$. When $p = q =1$ and $\mu =1$, we have $\alpha/\delta = 1$. In this case, if there are infinitely many $\lambda \in \C$ such that 
 $$ \alpha^n\lambda + \frac{1- \alpha^n}{1- \alpha}\beta = f^n(\lambda) = g^n(\lambda) = \alpha^n \lambda + \frac{1 - \alpha^n}{1 - \alpha}\gamma$$
 for some $n \in \N^+$, then we have
 $\beta = \gamma$ and hence $f =g$ contradicts the assumption that $f$ and $g$ generate a free semigroup.
 
 Now, we assume $p > 2$. Also, if $p = q$ then by the coprime assumption of $p$ and $q$, we have $p =q =1$. So, we have $p > 2$ and $p > q$ under the assumption of $p > 2$. Then Equation \ref{eq: poly-n-noroot-1}
 implies that 
 \begin{equation}\label{eq: poly-non-root-2}
     \mu A(x)^p \left(x - \frac{\beta}{1 - \alpha}\right) + B(x)^p\frac{\beta}{1 - \alpha} = A(x)^qB(x)^{p-q}\left(x - \frac{\gamma}{1 - \delta}\right) + \frac{\gamma}{1 - \delta}B(x)^p.
 \end{equation}
 Then $B(x)$ cannot be a constant due to the degree reason. But note that then there exists a $\nu \in \C$ such that $x - \nu$ divides $B(x)$ but not $A(x)$, so we have Equation (\ref{eq: poly-non-root-2}) implies that 
 $$x - \nu ~|~  \mu A(x)^p\left(x - \frac{\beta}{1 - \alpha}\right).$$
 Thus $\nu = \beta/(1 - \alpha)$. Now, Equation (\ref{eq: poly-non-root-2})
 implies that 
 \begin{align}\label{eq: poly-noroot-3}
     A(x)^qB'(x)^{p-q}\left(x - \frac{\gamma}{1 - \delta}\right)\left(x - \frac{\beta}{1 - \alpha}\right)^{p-q-1} + \frac{\gamma}{1 - \delta}B'(x)^p\left(x - \frac{\beta}{1 - \alpha}\right)^{p-1}\nonumber\\ = \mu A(x)^p + B'(x)^p\left(x - \frac{\beta}{1 - \alpha}\right)^{p-1}\frac{\beta}{1-  \alpha},
 \end{align}
 where $B(x) = (x -\beta/(1 - \alpha))B'(x)$.

 Note that if $\deg(B'(x)) > 0$, then there exists another $\nu_1 \in \C$ such that $x - \nu_1$ divides $B'(x)$ but not $A(x)$ and contradicts Equation (\ref{eq: poly-noroot-3}). 

 Now, assume $\deg(B'(x)) = 0$. If $\deg(A(x)) > 0$, then, similarly, there exists a linear factor $x - \nu_2$ dividing $A(x)$ but not $B(x)$, we have Equation (\ref{eq: poly-noroot-3}) implies that 
 $$\frac{\gamma}{1 - \delta}B'(x)^p\left(x - \frac{\beta}{1 - \alpha}\right)^{p-1} \equiv \frac{\beta}{1-  \alpha}B'(x)^p\left(x - \frac{\beta}{1 - \alpha}\right)^{p-1} \pmod{ x - \nu_2},$$
 which implies that 
 $$ \frac{\beta}{1 - \alpha} = \frac{\gamma}{1 - \delta}.$$
 But this implies $f$ and $g$ share two common fixed points, contradicting our assumption.

 Lastly, suppose $\deg(A(x)) = 0$, then since $p > 2$, the leading degree terms of Equation (\ref{eq: poly-noroot-3}) give again that 
 $$ \frac{\gamma}{1 - \delta}B'(x)^p\left(x - \frac{\gamma}{1 - \delta}\right)^{p-1} = \frac{\beta}{1-  \alpha}B'(x)^p\left(x - \frac{\beta}{1 - \alpha}\right)^{p-1},$$
 which implies 
 $$\frac{\beta}{1 - \alpha} = \frac{\gamma}{1 - \delta}. $$ Again, this is a contradiction. So, in this case, having infinitely many $\lambda$ making Equation (\ref{eq: targeteq1}) hold for some $n \in \N^+$ implies that $f$ and $g$ satisfy the exceptional condition (2).
 \newline

{\bf Case \RNum{2}:} Now, let us suppose $f(x)$ and $g(x)$ do not share a common fixed point. Then, under a suitable choice of coordinate, we have 
$$ f(x) = \alpha x + \beta$$
$$ g(x) = \frac{x}{\gamma\
x + \delta}$$
where $\gamma$, $\beta$, $\alpha$, $\delta \in \C^*$. \\

Similarly as in Case \RNum{1}, under the assumption that the two exceptional cases (1) and (2) do not occur, Remarks \ref{rmk: n=m-no-degen} and \ref{rmk: n=m-nondeg-2} indicates that we can apply Propositions \ref{prop: non-poly-m-n-root} and \ref{prop: rational-trans-multi-nocommon} without assuming $\{f, c\}$ and $\{g,c\}$ are non-degenerate, except when one of $\alpha$, $\beta$ is a root of unity other than $1$. If, without loss of generality, $\alpha$ is a root of unity other than $1$, then there exists a $m \in \N^+$ such that $f^m(x) = x$ which contradicts that $f$ and $g$ generate a free semigroup under compositions. So, we rule out this case and Propositions \ref{prop: non-poly-m-n-root} and \ref{prop: rational-trans-multi-nocommon} is applicable.  \\

Again, if one of $\alpha$ or $\delta$ is a root of unity other than $1$, then it contradicts the freeness assumption as $\{f^m : m\in \N^+\}$ or $\{g^n : n \in \N^+\}$ is a finite set. 

Suppose one of $\alpha$ and $\delta$ is $1$. Note that the case that they are both roots of unity is covered by the exceptional condition (1). We just need to show that if exactly one of them is $1$, then there are only finitely many $\lambda \in \C$ making the Equation (\ref{eq: targeteq1}) hold for some $n\in \N^+$. With a further change of coordinates, if necessary and interchange of the roles of $f$ and $g$, we can always assume that it is $\alpha = 1$. This case is covered by Proposition \ref{prop: rational-trans-multi-nocommon} which shows that there are only finitely many $\lambda \in \C$ satisfies the condition and we are done.

Now, suppose that neither $\alpha$ nor $ \delta$ is a root of unity. Then Proposition \ref{prop: non-poly-m-n-root} implies that having infinitely many $\lambda$ such that Equation (\ref{eq: targeteq1}) holds for some $n \in \N^+$ will result in 
\begin{equation}\label{eq: non-poly-case-1-contra}
    \left(x - \frac{\beta}{1 - \alpha}\right) \mu F(x)^p + \frac{\beta}{1 - \alpha} = x/\left( F(x)^q (1 - \frac{\gamma}{1 - \delta}x) + \frac{\gamma}{1 - \delta}x\right)
\end{equation}
where $p,q$ is a pair of coprime non-zero integers, $\mu^q = \alpha^{nq}/\delta^{np}$ for infinitely many $n \in \N$ and $F(x) = A(x)/B(x)$ for a pair of coprime polynomials $A(x), B(x) \in \C[x]$. Without loss of generality, we can always assume that $q > p$ by swapping the role of $f$ and $g$ if necessary. \\

{\bf Subcase (1):} Let us first suppose $p > 0$. Then Equation (\ref{eq: non-poly-case-1-contra}) is equivalent to 
\begin{align}\label{eq: non-poly-q>0-1}
    \left(x - \frac{\beta}{1- \alpha}\right)\left(1 - \frac{\gamma}{ 1- \delta}x\right)\mu A(x)^{p+q}+ \frac{\beta}{1 - \alpha}\left(1 - \frac{\gamma}{1 - \delta}x\right)A(x)^q B(x)^p \nonumber\\+ \frac{\gamma}{1 - \delta}\left(x - \frac{\beta}{1 - \alpha}\right)x\mu A(x)^p B(x)^q + \frac{\beta}{1 - \alpha}\frac{\gamma}{1 - \delta}x B(x)^{ p + q} = x B(x)^{p + q}.
\end{align}
We want to show that this cannot be true unless $p = q = 1$. We assume that $q > p \geq 1$ and try to get contradictions. Suppose $B(x)$ is a constant, then 
\begin{align}
    \deg\left(\left(x - \frac{\beta}{1- \alpha}\right)\left(1 - \frac{\gamma}{ 1- \delta}x\right)\mu A(x)^{p+q}\right) \nonumber\\> \deg\bigg(\frac{\beta}{1 - \alpha}\left(1 - \frac{\gamma}{1 - \delta}x\right)A(x)^q B(x)^p + \frac{\gamma}{1 - \delta}\left(x - \frac{\beta}{1 - \alpha}\right)x\mu A(x)^p B(x)^q \nonumber \\+ \frac{\beta}{1 - \alpha}\frac{\gamma}{1 - \delta}x B(x)^{ p + q} - x B(x)^{p + q}\bigg),
\end{align} 
since $q > p \geq 1$ and $\deg(A(x)) > 0$ as $F(x)$ is not a constant. This gives a contradiction.

Suppose $B(x)$ is not a constant but $A(x)$ is a constant. Then, Equation (\ref{eq: non-poly-q>0-1}) implies $\deg(B(x)) = 1$ and $p = 1$. Suppose $B(x) = \mu_b(x - \beta/(1- \alpha))$ and $A(x) = \mu_a$, where $\mu_b, \mu_a \in\C^*$. Then, Equation (\ref{eq: non-poly-q>0-1}) further implies, by grouping $(x - \beta/(1 - \alpha))$ terms, that
$$ \left(1 - \frac{\gamma}{1 - \delta}x\right)\mu \mu^{q+1}_a + \frac{\beta}{1 - \alpha} \left(1 - \frac{\gamma}{1 - \delta}x\right)\mu_a^q\mu_b \equiv 0 \pmod{x - \frac{\beta}{1 - \alpha}}.$$
Since $\beta/(1 - \alpha) \neq (1 - \delta)/\gamma$ as we assumed that $f$ and $g$ don't share a common fixed point, we must have 
\begin{equation}\label{eq: non-poly-q>0-2}
    \mu \mu_a + \frac{\beta}{1 - \alpha} \mu_b = 0.
\end{equation}
Then, Equation (\ref{eq: non-poly-q>0-1}) will further imply that 
\begin{equation}\label{eq: non-polu-q>0-3}
    \frac{\gamma}{1 - \delta} \mu \mu_a + \left(\frac{\gamma}{1 - \delta}\frac{\beta}{1 - \alpha} -1\right)\mu_b = 0,
\end{equation}
by looking at the rest of terms. Then Equation (\ref{eq: non-poly-q>0-2}) and (\ref{eq: non-polu-q>0-3}) together implies that $\mu_b = 0$, which is a contradiction. 

Then, suppose $B(x) = \mu_b (x - \gamma/(1 - \delta))$. Then, moding out $x - \gamma/(1 - \delta)$, Equation (\ref{eq: non-poly-q>0-1}) implies that 
$$ \left(x - \frac{\beta}{1 - \alpha}\right)\mu \mu_a^{q +1} = 0$$
which is impossible. 

Now, suppose $$B(x) = \mu_b (x - \nu)$$
where $\nu \in \C \setminus\left\{\frac{\beta}{1 - \alpha}, \frac{1 - \delta}{\gamma}\right\}$. Moding out $(x- \nu)$ in Equation (\ref{eq: non-poly-q>0-1}) gives that 
$$ \left(x - \frac{\gamma}{1- \delta}\right)\left(x - \frac{\beta}{1 - \alpha}\right)\mu \mu_a^{q +1} \equiv 0 \pmod{x - \nu},$$
which is only possible when $\mu\mu_a^{q+1} = 0$ contradicting our assumption.

Now, suppose both $B(x)$ and $A(x)$ are not constant. Then there exists a $\nu \in \C$ such that $x - \nu$ divides $B(x)$ but not $A(x)$. If $\nu \notin \{ \beta/(1 - \alpha), (1- \delta)/\gamma\}$, then, moding out $x - \nu$ in Equation (\ref{eq: non-poly-q>0-1}), we have that
$$ \mu A(x)^{p + q} \equiv 0 \pmod{x - \nu}$$
which contradicts our assumption that $x - \nu$ doesn't divide $A(x)$. 

Suppose $\nu = \beta /(1 - \alpha)$. Then Equation (\ref{eq: non-poly-q>0-1}) implies that 
\begin{align} \label{eq: non-poly-q>0-4}
   \left(1 - \frac{\gamma}{ 1- \delta}x\right)\mu A(x)^{p+q}+ \frac{\beta}{1 - \alpha}\left(1 - \frac{\gamma}{1 - \delta}x\right)A(x)^q B'(x)^p\left(x - \frac{\beta}{1 - \alpha}\right)^{p-1} \nonumber \\+ \frac{\gamma}{1 - \delta}x\mu A(x)^p B(x)^q + \left(\frac{\beta}{1 - \alpha}\frac{\gamma}{1 - \delta} -1 \right)x B'(x)^{ p + q}\left(x - \frac{\beta}{1- \alpha}\right)^{p + q - 1} =0,
\end{align}
where $B(x) = B'(x) (x -\beta/(1 - \alpha)).$ Then let $x - \nu'$ be a linear factor dividing $A(x)$ but not $B(x)$, we have Equation (\ref{eq: non-poly-q>0-4}) implies that 
$$\left(\frac{\beta}{1 - \alpha}\frac{\gamma}{1 - \delta} -1 \right)x B'(x)^{ p + q}\left(x - \frac{\beta}{1- \alpha}\right)^{p + q - 1} \equiv 0 \pmod{x - \nu'} ,$$
which is only possible when $\nu' = 0$ since $\beta/(1 - \alpha) \neq (1 - \delta)/\gamma$. However, if $\nu' = 0$, then Equation (\ref{eq: non-poly-q>0-4}) will imply that 
\begin{align}
     \left(1 - \frac{\gamma}{ 1- \delta}x\right)\mu x^{p+ q-1}A'(x)^{p+q}+ \frac{\beta}{1 - \alpha}\left(1 - \frac{\gamma}{1 - \delta}x\right)x^{p-1}A'(x)^q B'(x)^p\left(x - \frac{\beta}{1 - \alpha}\right)^{p-1} \nonumber \\+ \frac{\gamma}{1 - \delta}x^p\mu A'(x)^p B(x)^q + \left(\frac{\beta}{1 - \alpha}\frac{\gamma}{1 - \delta} -1 \right) B'(x)^{ p + q}\left(x - \frac{\beta}{1- \alpha}\right)^{p + q - 1} =0,
\end{align}
where $A(x) = xA'(x)$. Again, this implies that 
$$ \left(\frac{\beta}{1 - \alpha}\frac{\gamma}{1 - \delta} -1 \right) B'(x)^{ p + q}\left(x - \frac{\beta}{1- \alpha}\right)^{p + q - 1} \equiv 0 \pmod{x} $$
which is a contradiction.

Lastly, if $\nu = (1 - \delta)/\gamma$, then after dividing a $(x - \nu)$ on both side of Equation (\ref{eq: non-poly-q>0-4}), we get
$$
\left(x - \frac{\beta}{1 - \alpha}\right)\mu A(x)^{p + q} \equiv 0 \pmod{x - \nu },
$$
which is also impossible.\\

{\bf Subcase (2):} Now, let us suppose $p < 0$. By choosing $F(x)$ and swapping the role of $f$ and $g$ if necessary, we can assume that $q \geq -p$. We want to show that Equation (\ref{eq: non-poly-case-1-contra}) doesn't hold unless $q =  -p = 1$. We assume that $q > -p\geq 1$ and try to get a contradiction. Note that by clearing the denominator, we have Equation (\ref{eq: non-poly-case-1-contra}) is equivalent to 
\begin{align}\label{eq: non-poly-q<0-1}
    \left(x - \frac{\beta}{1 - \alpha}\right)\left(1 - \frac{\gamma}{1 - \delta}x\right)\mu A(x)^q B(x)^{-p } + \frac{\beta}{1 - \alpha}\left(1 - \frac{\gamma}{1 - \delta}x\right)A(x)^{q -p} \nonumber \\ + \frac{\gamma}{1 - \delta}x\left(x - \frac{\beta}{1 - \alpha}\right)\mu B(x)^{q -p} + \left(\frac{\gamma\beta}{(1- \alpha)(1 - \delta)} - 1\right)xA(x)^{-p}B(x)^q = 0.
\end{align}
Suppose $B(x)$ is not constant, then there exists a $\nu_1 \in \C$ such that $x - \nu$ divides $B(x)$ but not $A(x)$. Suppose $\nu\neq (1 - \delta)/\gamma$. Then, moding out $x  - \nu$, Equation (\ref{eq: non-poly-q<0-1}) implies that
$$ \frac{\beta}{1 - \alpha}\left(1 - \frac{\gamma}{1- \delta}x\right)A(x)^{q-p} \equiv 0 \pmod{x - \nu},$$
which is a contradiction. While, if $\nu =(1 - \delta)/\gamma$, then Equation (\ref{eq: non-poly-q<0-1}) implies 
\begin{align}
  \left(x - \frac{\beta}{1 - \alpha}\right)\mu A(x)^q B(x)^{-p } + \frac{\beta}{1 - \alpha}A(x)^{q -p} \nonumber  \nonumber\\+ \frac{\gamma}{1 - \delta}x\left(x - \frac{\beta}{1 - \alpha}\right)\mu B'(x)^{q -p}\left(1 - \frac{\gamma}{1- \delta}x\right)^{q-p-1} \\+ \left(\frac{\gamma\beta}{(1- \alpha)(1 - \delta)} - 1\right)xA(x)^{-p}B'(x)^q\left(1 - \frac{\gamma}{1 - \delta}x\right)^{q-1} = 0, 
\end{align}
which again gives contradiction after moding out $x - (1- \delta)/\gamma$, here $B(x) = B'(x)(1 - \gamma x/(1 - \delta))$.

Now, suppose $B(x) = \mu_b \in \C^*$ is a constant. Then since $F(x)$ is not a constant, we must have $A(x)$ is not a constant. Then there exists a $\nu_1 \in \C$ such that $x - \nu_1$ divides $A(x)$. If $\nu_1 \neq \beta/(1 - \alpha)$ or $p \neq -1$, then Equation (\ref{eq: non-poly-q<0-1}) gives a contradiction after moding out $x - \nu_1$. So, we assume that $\nu_1 = \beta/(1 - \alpha)$ and $p = -1$. 
We write $A(x) = A'(x) (x - \beta/(1 - \alpha))$. Then Equation (\ref{eq: non-poly-q<0-1}) implies that 
\begin{align}\label{eq: non-poly-q<0-2}
    \left(1 - \frac{\gamma}{1 - \delta}x\right)\mu A'(x)^q \mu_b\left(x - \frac{\beta}{1 - \alpha}\right)^p + \frac{\beta}{1 - \alpha}\left(1 - \frac{\gamma}{1 - \delta}x\right)A'(x)^{q +1}\left(x - \frac{\beta}{1 - \alpha}\right)^{q} \nonumber \\ + \frac{\gamma}{1 - \delta}x\mu \mu_b^{q +1} + \left(\frac{\gamma\beta}{(1- \alpha)(1 - \delta)} - 1\right)xA'(x)\mu_b^q = 0.
\end{align}
If there exists a $\nu_2 \in \C^* $ such that $x - \nu_2 $ divides $A'(x)$, then again, moding out $x - \nu_2$ in Equation (\ref{eq: non-poly-q<0-2}), we have 
$$ \frac{\gamma}{1 - \delta}\mu\mu_b^{q+1} = 0,$$
which is a contradiction.
If $x$ divides $A'(x)$, then after dividing both side of Equation (\ref{eq: non-poly-q<0-2}) by $x $, we again have that, after moding out $x$, 
$$\frac{\gamma}{1 - \delta}\mu\mu_b^{q+1} = 0,$$
which is a contradiction.

So, we may assume that $A'(x) = \mu_a \in \C^*$ is a constant. Then Equation (\ref{eq: non-poly-q<0-2}) implies that 
\begin{align}
     \left(1 - \frac{\gamma}{1 - \delta}x\right)\mu \mu_a^q \mu_b\left(x - \frac{\beta}{1 - \alpha}\right)^q + \frac{\beta}{1 - \alpha}\left(1 - \frac{\gamma}{1 - \delta}x\right)\mu_a^{q +1}\left(x - \frac{\beta}{1 - \alpha}\right)^{q} \nonumber \\ + \frac{\gamma}{1 - \delta}x\mu \mu_b^{q +1} + \left(\frac{\gamma\beta}{(1- \alpha)(1 - \delta)} - 1\right)x\mu_a\mu_b^q = 0.
\end{align}
This implies that 
$$ \mu\mu_b + \frac{\beta}{1 - \alpha}
\mu_a = 0,$$
$$ \frac{\gamma}{1 - \delta}\mu\mu_b + \left(\frac{\gamma\beta}{(1 - \alpha)(1 - \delta)} -1\right)\mu_a=0.$$
Together, we have $\mu_a = 0$ which is a contradiction.

\end{proof}

\subsection{Positive Characteristic Case}

Let the base field $K$ be an algebraically closed field of characteristic $p>0$.

\begin{lem}
Let $\alpha,\delta\in K\backslash\overline{\mathbb{F}_p}$ and let $C\subseteq\mathbb{G}_m^2$ be an irreducible closed curve.
\begin{enumerate}
\item
$C$ has an infinite intersection with the cyclic group $\{(\alpha^n,\delta^n)|\ n\in\mathbb{Z}\}$ if and only if there exists
\begin{enumerate}
\item
$e_1\in\mathbb{Z},e_2\in\mathbb{Z}\backslash\{0\}$, and $q$ which is a power of $p$,
\item
$\alpha',\delta'\in K$ such that $\alpha'^{q-1}=\alpha$ and $\delta'^{q-1}=\delta$, and
\item
an absolutely irreducible polynomial $F(x,y)\in\mathbb{F}_q[x,y]$,
\end{enumerate}
such that $F(\alpha'^{e_2},\delta'^{e_2})=0$ and $C$ is defined by the formula
$$
F\left(\frac{\alpha'^{e_2}\cdot x}{\alpha^{e_1}},\frac{\delta'^{e_2}\cdot y}{\delta^{e_1}}\right)=0.
$$

\item
$C$ has an infinite intersection with the abelian group $\{(\alpha^m,\delta^n)|\ m,n\in\mathbb{Z}\}$ if and only if there exists
\begin{enumerate}
\item
$(m_0,n_0)\in\mathbb{Z}^2,(m_1,n_1)\in\mathbb{Z}^2\backslash\{(0,0)\}$, and $q$ which is a power of $p$,
\item
$\alpha',\delta'\in K$ such that $\alpha'^{q-1}=\alpha$ and $\delta'^{q-1}=\delta$, and
\item
an absolutely irreducible polynomial $F(x,y)\in\mathbb{F}_q[x,y]$,
such that $F(\alpha'^{m_1},\delta'^{n_1})=0$ and $C$ is defined by the formula
$$
F\left(\frac{\alpha'^{m_1}\cdot x}{\alpha^{m_0}},\frac{\delta'^{n_1}\cdot y}{\delta^{n_0}}\right)=0.
$$
\end{enumerate}
\end{enumerate}
\end{lem}

\begin{proof}
\begin{enumerate}
\item
Firstly, suppose the intersection is infinite. Then according to \cite[Theorem B]{MS04}, there exist $q$ and $e_1,e_2$ satisfying condition (a) such that $$\left\{e_1+e_2\frac{q^m-1}{q-1}|\ m\in\mathbb{N}\right\}\cdot(\alpha,\delta)\subseteq C.$$ We pick $\alpha',\delta'\in K$ such that $(q-1)\cdot(\alpha',\delta')=(\alpha,\delta)$. Then $$\{(e_1q-e_1-e_2)+e_2q^m|\ m\in\mathbb{N}\}\cdot(\alpha',\delta')\subseteq C.$$ Hence for any $m\geq0$, we have $q^m\cdot(\alpha'^{e_2},\delta'^{e_2})\in C'$, where $$C'=(\alpha'^{e_1+e_2-e_1q},\delta'^{e_1+e_2-e_1q})+C.$$

Let $[q]$ be the multiple-by-$q$ map on $\mathbb{G}_m^2$. Then we see that $[q](C')=C'$. Thus $C'$ is defined over $\mathbb{F}_q$. Let $F(x,y)\in\mathbb{F}_q[x,y]$ be the defining equation of $C'$. Then $F(\alpha'^{e_2},\delta'^{e_2})=0$, and $C$ is defined by the formula $$F\left(\frac{\alpha'^{e_2}\cdot x}{\alpha^{e_1}},\frac{\delta'^{e_2}\cdot y}{\delta^{e_1}}\right)=0.$$

For another side, suppose the data $\alpha,\delta$ and $C$ satisfy the conclusion. Then $$\left\{e_1+e_2\frac{q^m-1}{q-1}|\ m\in\mathbb{N}\right\}\cdot(\alpha,\delta)\subseteq C.$$ Hence the intersection is infinite.

\item
The proof is similar as above. If the intersection is infinite, then for the same reason \cite[Theorem B]{MS04}, there exist $q$ and $m_0,n_0,m_1,n_1$ satisfying condition (a) such that $$(\alpha^{m_0},\delta^{n_0})+\left\{\frac{q^m-1}{q-1}|\ m\in\mathbb{N}\right\}\cdot(\alpha^{m_1},\delta^{n_1})\subseteq C.$$ Pick $\alpha',\delta'\in K$ such that $(q-1)\cdot(\alpha',\delta')=(\alpha,\delta)$. Then $(\alpha^{m_0},\delta^{n_0})-(\alpha'^{m_1},\delta'^{n_1})+\{q^m|\ m\in\mathbb{N}\}\cdot(\alpha'^{m_1},\delta'^{n_1})\subseteq C$. Same as above, the curve $C':=(\alpha'^{m_1},\delta'^{n_1})-(\alpha^{m_0},\delta^{n_0})+C$ is defined over $\mathbb{F}_q$. Let $F(x,y)\in\mathbb{F}_q[x,y]$ be the defining equation of $C'$. Then $F(\alpha'^{m_1},\delta'^{n_1})=0$, and $C$ is defined by the formula $$F\left(\frac{\alpha'^{m_1}\cdot x}{\alpha^{m_0}},\frac{\delta'^{n_1}\cdot y}{\delta^{n_0}}\right)=0.$$

For another side, suppose the data $\alpha,\delta$ and $C$ satisfy the conclusion. Then $$(\alpha^{m_0},\delta^{n_0})+\left\{\frac{q^m-1}{q-1}|\ m\in\mathbb{N}\right\}\cdot(\alpha^{m_1},\delta^{n_1})\subseteq C.$$ Hence the intersection is infinite.
\end{enumerate}
\end{proof}

 We firstly make an observation. If $f$ and $g$ admit a common fixed point, then we move that point to infinity by a conjugation, and then $f$ and $g$ become linear polynomials. Otherwise, we may conjugate two fixed points of $f$ and $g$ to $\infty$ and 0, respectively. Then $f$ will become a linear polynomial and $g$ will have the form $g(x)=x/(\gamma x + \delta)$.

In view of the reduction step above, the following theorem essentially answers the Question \ref{qu: automorphisms-common-zeros} when $K$ is positive characteristic.

\begin{thm}[Theorem \ref{thm: main-classification-positive-char}]
Let $f(x)=\alpha x+\beta$ and $g(x)=\delta x+\gamma$, where $\alpha,\delta\in K\backslash\overline{\mathbb{F}_p}$ and $\beta,\gamma\in K$. Let $c_1(x),c_2(x)\in K(x)$ be rational functions such that $c_1 \not\in \{f^n : n \in \Z\}$ and $c_2 \not \in \{g^n : n \in \Z\}$. Write $$C_1(x)=\frac{c_1(x)-\frac{\beta}{1-\alpha}}{x-\frac{\beta}{1-\alpha}}$$ and $$C_2(x)=\frac{c_2(x)-\frac{\gamma}{1-\delta}}{x-\frac{\gamma}{1-\delta}}.$$

\begin{enumerate}
\item
$\bigcup\limits_{n\in\mathbb{Z}}\left\{x\in K|\ f^n(x)=c_1(x),g^n(x)=c_2(x)\right\}$ is an infinite set if and only if $C_1$ and $C_2$ are non-constant, and there exist
\begin{enumerate}
\item
$e_1\in\mathbb{Z},e_2\in\mathbb{Z}\backslash\{0\}$, and $q$ which is a power of $p$,
\item
$\alpha',\delta'\in K$ such that $\alpha'^{q-1}=\alpha$ and $\delta'^{q-1}=\delta$, and
\item
an absolutely irreducible polynomial $F(x,y)\in\mathbb{F}_q[x,y]$,
\end{enumerate}
such that $F(\alpha'^{e_2},\delta'^{e_2})=0$ and $$F\left(\frac{\alpha'^{e_2}}{\alpha^{e_1}}\cdot C_1(x),\frac{\delta'^{e_2}}{\delta^{e_1}}\cdot C_2(x)\right)$$ is identically zero.

\item
$\bigcup\limits_{(m,n)\in\mathbb{Z}^2}\left\{x\in K|\ f^m(x)=c_1(x),g^n(x)=c_2(x)\right\}$ is an infinite set if and only if $C_1$ and $C_2$ are non-constant, and there exist
\begin{enumerate}
\item
$(m_0,n_0)\in\mathbb{Z}^2,(m_1,n_1)\in\mathbb{Z}^2\backslash\{(0,0)\}$, and $q$ which is a power of $p$,
\item
$\alpha',\delta'\in K$ such that $\alpha'^{q-1}=\alpha$ and $\delta'^{q-1}=\delta$, and
\item
an absolutely irreducible polynomial $F(x,y)\in\mathbb{F}_q[x,y]$
\end{enumerate}
such that $F(\alpha'^{m_1},\delta'^{n_1})=0$ and $$F\left(\frac{\alpha'^{m_1}}{\alpha^{m_0}}\cdot C_1(x),\frac{\delta'^{n_1}}{\delta^{n_0}}\cdot C_2(x)\right)$$ is identically zero.
\end{enumerate}

If $$g(x)=\frac{x}{\gamma x+\delta},$$ then we write $$C_2(x)=x\cdot\frac{\frac{1}{c_2(x)}-\frac{\gamma}{1-\delta}}{1-\frac{\gamma x}{1-\delta}}$$ and the same statements hold.
\end{thm}

\begin{proof}
The proposition has four parts. We shall only prove part (1) for $g(x)=\delta x+\gamma$, as the proof of other parts are just the same.

Firstly, suppose the set is infinite. We rewrite the conditions $f^n(x)=c_1(x)$ and $g^n(x)=c_2(x)$ as $\alpha^n=C_1(x)$ and $\delta^n=C_2(x)$. The assumptions on $c_1$ and $c_2$ then guarantee that $C_1$ and $C_2$ are non-constant. Let $C\subseteq\mathbb{G}_m^2$ be the closed image of the rational map $\mathbb{P}^1\dashrightarrow\mathbb{G}_m^2$ given by $x\dashrightarrow(C_1(x),C_2(x))$. Then this curve has an infinite intersection with the cyclic group $\mathbb{Z}\cdot(\alpha,\delta)$. Hence we conclude by Lemma 1(1).

For another side, suppose the conclusion is satisfied. Then Lemma 1(1) says that the curve $C$ as above has an infinite intersection with $\mathbb{Z}\cdot(\alpha,\delta)$. Hence the set is infinite.
\end{proof}

\begin{rmk}
The hypotheses on $\alpha,\delta$ and $c_1,c_2$ in the proposition above are just to rule out some degenerate cases and simplify the conclusion. The arguments for these degenerate cases are essentially easy but somehow tedious, so we leave them to interested readers.
\end{rmk}

Using our description, one can give a negative answer of the first part of \cite[Question 17]{HT17}. We firstly restate the question in a slightly more general and seemingly more natural way.

\begin{question}(first part of \cite[Question 17]{HT17})
Let $f$ and $g$ be two compositionally independent non-isotrivial polynomials in $K[x]$, and let $c\in K[x]$. Is it true that there are at most finitely many $\lambda\in K$ such that there is an $n$ for which $(x-\lambda)$ divides $\mathrm{gcd}(f^n(x)-c(x),g^n(x)-c(x))$?
\end{question}

\begin{rmk}
\begin{enumerate}
\item
Our above discussion about the problem has nothing to do with the compositional dependence of $f$ and $g$. Indeed, this requirement comes from the problem for higher degree endomorphisms (in characteristic 0).

\item
In the original question, the authors consider the problem over the base ring $\mathbb{F}_q[T]$. One can see that our counterexamples below also give a negative answer to the original question.
\end{enumerate}
\end{rmk}

Now we give counterexamples of the question. We consider linear polynomials $f$ and $g$.

\begin{example}
Let $q$ be a power of $p$. Let $a\in\mathbb{F}_q^{\times},t\in K\backslash\overline{\mathbb{F}_p}$ and let $h(x)\in\mathbb{F}_q[x]$ be a non-constant polynomial. Let $f(x)=(th(t)+1)x+ath(t)$ and $g(x)=((t+a)h(t)+1)x$. Fortunately, they are indeed compositionally independent. Let $c(x)=x((x+a)h(x)+1)$. Let $\lambda_m=t^{q^m}$ for every $m\geq0$. Then one can see that $$(x-\lambda_m)\mid\mathrm{gcd}\left(f^{q^m}(x)-c(x),g^{q^m}(x)-c(x)\right)$$ for every $m$.
\end{example}

\section{Connection to the Generalized Dynamical Mordell-Lang Problem}\label{sec: connect-GDML-group}
In this section and the next, we will show that for some well studied dynamical systems $(X,f,g,c)$, the set of positive integers 
$$\{n : f^n(x) = g^n(x) = c(x) \text{ admit solutions}\}$$
is a finite union of arithmetic progression, which, as explained in the introduction, is a special case of the generalized Dynamical Mordell-Lang Problem (Question \ref{qu: DML-varieties}) asked in \cite[Question 9.13(i)]{Xie23}.

We repeat the Question \ref{qu: DML-varieties} here for the convenience of the reader: 
\begin{question}[Question \ref{qu: DML-varieties}]
Let $X$ be a variety and let $F$ be an endomorphism of $X$. Let $Z$ and $V$ be irreducible closed subvarieties of $X$ such that $\mathrm{dim}(Z)+\mathrm{dim}(V)<\mathrm{dim}(X)$. Then can we describe the set $\{n\in\mathbb{N}|\ F^n(Z)\cap V\neq\emptyset\}$? If the characteristic is 0, is it a finite union of arithmetic progressions?
\end{question}

We will answer this particular special case of Question \ref{qu: DML-varieties} in the case that $f$, $g$ and $c$ are all polynomials and $X = \A^1_\C$.

Moreover, when $F$ is an automorphisms on the projective variety $X$ coming from an algebraic group action and everything is defined over an algebraically closed field $K$, we answer Question \ref{qu: DML-varieties} completely.\\

In this section, we will resolve the case when $F$ is an automorphism defined on the projective variety $X$ coming from an algebraic group. We will leave the case regarding polynomials to the next section.\\

It is worth remarking that when studying a problem in dynamical systems, the first---and often the most tractable---case to consider is when the dynamics arise from an algebraic group. Systems coming from group actions typically exhibit lower complexity, and key objects such as forward orbits of subvarieties and iterates of the maps often have bounded degree. For these reasons, it is natural to begin with the algebraic group case before turning to other special instances of Question~\ref{qu: DML-varieties}.\\

We shall focus on the projective case for simplicity. The general case can be tackled in the same way, but the conclusion will be more complicated. Please see Remark \ref{rmk: non-projective-GDML} for this.

For the positive characteristic case, we need the notion of ``widely $p$-normal sets". We omit the somehow complicated definition here and refer to \cite[Definition 1.1]{XY25}.

\begin{thm}[Theorem \ref{thm: main-group-GDML}] \label{thm: group-GDML}
Let $f$ be an automorphism of a projective variety $X$. Suppose that $f$ comes from an algebraic group action. Let $Z$ and $V$ be closed subvarieties of $X$. Consider the set $\{n\in\mathbb{Z}|\ f^n(Z)\cap V\neq\emptyset\}$.
\begin{enumerate}
\item
If $\mathrm{char}(K)=0$, then this set is a finite union of arithmetic progressions.
\item
If $\mathrm{char}(K)=p>0$, then this set is a widely $p$-normal set.
\end{enumerate}
\end{thm}

\begin{proof}
Write $f=\rho_g$ for an algebraic group action $\rho:G\times X\rightarrow X$ and a point $g\in G(K)$. Then $f^n=\rho_{g^n}$ and hence $f^n(Z)=\rho(\{g^n\}\times Z)$. So $$f^n(Z)\cap V\neq\emptyset$$ is equivalent to $$(\{g^n\}\times Z)\cap\rho^{-1}(V)\neq\emptyset.$$ If we denote $\rho_Z:=\rho|_{G\times Z}$, then this condition is furthermore equivalent to $$g^n\in\mathrm{pr}_G(\rho_Z^{-1}(V)),$$ where $\mathrm{pr}_Z$ is the projection $G\times Z\rightarrow G$.

Let $V_0\subseteq G$ be the closed subset $\mathrm{pr}_G(\rho_Z^{-1}(V))$, then we need to prove $\{n\in\mathbb{Z}|\ g^n\in V_0(K)\}$ has the expected form. If $\mathrm{char}(K)=0$, then this follows from \cite[Theorem 7]{CS93} or \cite{BGT10}. If $\mathrm{char}(K)=p>0$, then this is \cite[Theorem 3.1]{XY25}.
\end{proof}

\begin{rmk}\label{rmk: non-projective-GDML}
If we do not require $X$ to be projective, then the $V_0$ above will just be a constructible set. Hence we need to allow difference sets. For example, if $\mathrm{char}(K)=0$, then the conclusion will be a finite union of one-sided arithmetic progressions. Note that when we restrict to $n \in \N$, an arithmetic progression and a one-sided arithmetic progression coincide. 
\end{rmk}

\section{Special Cases Related to Common Zeros of Iterated Morphisms}\label{sec: GDML-poly}

In this section, we will resolve a special case of Question \ref{qu: DML-varieties} related to our common zeros of iterated morphisms question as mentioned in the previous section. Specifically, we will prove that the the set of $n \in \N^+$ such that 
$$ f^n(x) = g^n(x) = c(x)$$
admits solution is a finite union of arithmetic progressions when $f$, $g$ and $c$ are all polynomials and $X = \A^1_\C$.

It turns out that one of the key steps in our argument relies on an unexpectedly subtle result from elementary number theory, which may be of independent interest. Therefore, before proceeding to prove the relevant special case of Question~\ref{qu: DML-varieties}, we introduce and resolve this number-theoretic problem in a preliminary subsection.

\subsection{Preliminaries on elementary number theory}

The question asks, for a fixed root of unity $\xi$, natural numbers $d_1, d_2 > 1$ and $d_3, a,b,c_1,c_2 \in \N$, for which $n \in \N$ does a system of two equations involving power maps: 
   $$ (\xi^ax^{d_1})^{\circ n} = \xi^{c_1} x^{d_3}  $$
   $$ (\xi^bx^{d_2})^{\circ n} = \xi^{c_2} x^{d_3}$$
admit a common solution in $\C^*$. It turns out that this is equivalent to ask for which index $n$ does a specific greatest common divisor condition hold (see Lemma \ref{lem: power-reduce-to-gcd}). Then, we show that such a set of indices $n$ making the greatest common divisor condition hold is a finite union of arithmetic progressions (that is Proposition \ref{prop: gcd-arithmetic-progressions}). As a direct corollary, we answer this question about common solutions in Corollary \ref{prop: poly-power-arith}.

\begin{lem}\label{lem: power-reduce-to-gcd}
Let \( \xi \) be a primitive \( k \)-th root of unity, and let \( d_1, d_2 > 1 \) be integers.  
Let \( d_3, d_4 \ge 0 \), and let \( a, b \) be integers with \( 0 \le a, b < k \).  
For any \( n \in \mathbb{N} \) satisfying \( \min\{d_1^n, d_2^n\} > \max\{d_3, d_4\} \), there exists \( x_0 \in \mathbb{C} \) such that
\begin{equation}\label{eq: common-sol-power-1}
    x_0^{\, d_1^n - d_3} = \xi^a,
\end{equation}
and
\begin{equation}\label{eq: common-sol-power-2}
    x_0^{\, d_2^n - d_4} = \xi^b,
\end{equation}
if and only if
\[
    k \cdot \gcd(d_2^n - d_4,\; d_1^n - d_3)
    \mid b(d_1^n - d_3) - a(d_2^n - d_4).
\]

\end{lem}

\begin{proof}
Fix \( n \in \mathbb{N} \) such that \( \min\{d_1^n, d_2^n\} > \max\{d_3, d_4\} \).  
The existence of \( x_0 \in \mathbb{C} \) satisfying  
\eqref{eq: common-sol-power-1} and \eqref{eq: common-sol-power-2}  
is equivalent to the existence of a positive integer \( m \) and an integer \( y \) such that
\begin{align}
    (d_1^n - d_3)y &\equiv am \pmod{km}, \label{eq: cong1}\\
    (d_2^n - d_4)y &\equiv bm \pmod{km}. \label{eq: cong2}
\end{align}

Equivalently, there exist integers \( z_1, z_2,y \) and positive integer $m$ such that
\begin{align*}
    (d_1^n - d_3)y &= am + km z_1,\\
    (d_2^n - d_4)y &= bm + km z_2.
\end{align*}
Hence,
\[
    \frac{a + kz_1}{d_1^n - d_3}
    = \frac{y}{m}
    = \frac{b + kz_2}{d_2^n - d_4}.
\]

Thus, a solution exists if and only if there are integers \( z_1, z_2 \) satisfying
\[
    (a + kz_1)(d_2^n - d_4) = (b + kz_2)(d_1^n - d_3),
\]
which can be rewritten as
\[
    k\bigl((d_2^n - d_4)z_1 - (d_1^n - d_3)z_2 \bigr)
    = b(d_1^n - d_3) - a(d_2^n - d_4).
\]

Such integers \( z_1, z_2 \) exist precisely when
\[
    k \cdot \gcd(d_2^n - d_4,\; d_1^n - d_3)
    \mid b(d_1^n - d_3) - a(d_2^n - d_4),
\]
establishing the desired criterion.
\end{proof}

\begin{prop}[Proposition \ref{prop: gcd-arithmetic-progressions}]

Let \( K \) be a number field with ring of integers \( \mathcal{O} \).  
Let \( \k \subseteq \mathcal{O} \) be an ideal and let \( d_1, d_2 \in \mathcal{O} \).  
Let \( d_3, a, b \in \mathcal{O} \).  
Then the set
\[
\left\{\, n \in \mathbb{N} :\; b(d_1^{n} - d_3) - a(d_2^{n} - d_3) \in 
\k\, (d_1^{n} - d_3,\; d_2^{n} - d_3) \,\right\}
\]
is a finite union of arithmetic progressions.

\end{prop}
\begin{proof}
The major case is when neither $d_1$ nor $d_2$ is zero or a root of unity. Before we start to discuss this major case, we first treat the relatively easy case when at least one of them is $0$ or a root of unity. Suppose one of them is $0$ and without loss of generality we assume that it is $d_1$, then for each prime $\p \mid \k$, we have $v_\p(-d_3, d^n_2 -d_3)$ is bounded by a natural number $l(\p)$. Also, for each $l' \leq l$, we have the set 
$$ L'_{l',\p} \coloneqq \{ n \in \N^+ : v_\p(-d_3, d^n_2 -d_3) \geq l'\}$$
is a finite union of arithmetic progressions since 
$$v_\p(-d_3, d^n_2 -d_3) \geq l'$$
is equivalent to 
$$ d^n_2 \equiv d_3 \pmod{\p^{l'}},$$
$$ d_3 \equiv 0 \pmod{\p^{l'}}.$$

Since the complement of a finite union of arithmetic progressions is again a finite union of arithmetic progressions, we also have 
$$ L_{l',\p} \coloneqq \{ n \in \N^+ : v_\p(-d_3, d^n_2 -d_3) = l'\}$$
is a finite union of arithmetic progressions.

Now, we have 
\begin{align}
  \left\{\, n \in \mathbb{N} :\; b(d_1^{n} - d_3) - a(d_2^{n} - d_3) \in 
\k\, (d_1^{n} - d_3,\; d_2^{n} - d_3) \,\right\}  \\
= \bigcap_{\p \mid \k} \bigcup_{0 \leq l' \leq l(\p)} \{n \in L_{l',\p}: -ad_3 +b(d^n_2 - d_3) \equiv 0 \pmod{\p^{l' + v_\p(\k)}} \},
\end{align} 
which is again a finite union of arithmetic progressions.\\

Now, suppose one of $d_1$ and $d_2$ is a root of unity and again we assume it is $d_1$. Note that if there exists a $n \in \N$ such that $d^n_1 = d_3$, then it holds for any $n \in L \coloneqq \{ms + t: \forall m \in \N \}$ for some $s,t \in \N$. Then the set 
$$ \{n \in L : b(d_1^{n} - d_3) - a(d_2^{n} - d_3) \in 
\k\, (d_1^{n} - d_3,\; d_2^{n} - d_3)\} $$ 
$$= \{n \in L : -a(d^n_2 -d_3) \in \k(d^n_2-d_3)\} = L.$$

When $n \in \N \setminus L$, we again have $v_\p(d^n_1 -d_3, d^n_2 -d_3)$ is bounded by a natural number $l(\p)$ for each $\p \mid \k$, since $d^n_1-d_3$ takes only finitely many possible values. Therefore, repeating the argument above for the case that $d_1 = 0$, we again obtain that in this case the set of $n$ satisfying the condition is a finite union of arithmetic progressions.\\

   Now, we begin to discuss the major case when $d_1, d_2$ are assumed to be non-zero and not root of unity. We first handle the special case that there exists a $m \in \N^+$ such that $d^m_1 = d^m_2$. Then obviously we only need to verify that for any $i \in \{1,2, \dots , m-1\}$, the set 
   \[\left\{\, n \in L_i :\; b(d_1^{n} - d_3) - a(d_2^{n} - d_3) \in 
\k\, (d_1^{n} - d_3,\; d_2^{n} - d_3) \,\right\}\]
is a finite union of arithmetic progressions, where $L_i \coloneqq \{mn + i: n \in \N\}$. It is sufficient to prove that for each prime ideal $\p \mid \k$, the set of $n \in L_i$ satisfying 
\[ b(d_1^n - d_3) - a(d_2^n - d_3)
    \in \mathfrak{p}^{\, v_{\mathfrak{p}}(\k)} (d_1^n - d_3,\; d_2^n - d_3)\]
is a finite union of arithmetic progressions. Note that for any $i \in \{1,2, \dots, m-1\}$ and $n \in L_i$, we have $d^n_2 = \mu_i d^n_1$ where $\mu_i \neq 1$ is a root of unity of order $m$. Then there exists a $r \in \N$ such that 
\[v_\p (d^n_1 - d_3, d^n_2 -d_3)  \leq r\]
for any $n \in L_i$. This is because 
$$ v_p(d^n_1 - d_3, d^n_2 -d_3) = v_p(d^n_1 - d_3, \mu_id^n_1 - d_3) = v_p(d^n_1 -d_3,d^n_1 - \mu^{-1}_id_3) \leq v_p((1 - \mu_i^{-1})d_3).$$
Also, note that for any $r' \in \N^+$, the set of $n \in \N^+$ such that 
$$ d^n_1 \equiv d_3 \pmod{\p^{r'}}$$
$$d^n_2 \equiv d_3 \pmod{\p^{r'}}$$
is a finite union of arithmetic progressions. Since the complement of a finite union of arithmetic progressions is again a finite union of arithmetic progressions, we have that for any $r' \in \{0,1,2,\dots, r\}$ the set of $n \in \N^+$ such that 
$$v_p(d^n_1 - d_3, d^n_2 -d_3) = r'$$
is a finite union of arithmetic progressions, denoted as $S_{r'}$. 

Therefore, we have 
 \[\left\{\, n \in L_i :\; b(d_1^n - d_3) - a(d_2^n - d_3)
    \in \mathfrak{p}^{\, v_{\mathfrak{p}}(\k)} (d_1^n - d_3,\; d_2^n - d_3) \,\right\}\]
\[ = \bigcup_{0 \leq r' \leq r}\left\{\, n \in L_i\cap S_{r'} :\; b(d_1^n - d_3) - a(d_2^n - d_3)
    \in \mathfrak{p}^{\, v_{\mathfrak{p}}(\k) + r'} \,\right\},\]
    which is a finite union of arithmetic progressions.\\

    From now on, we assume that there doesn't exists a $m \in \N^+$ such that $d^m_1 = d_2^m$.

    {\bf Case \RNum{1}:} We first consider the case \( d_3 \neq 0\) and is not a root of unity.  
Again, it suffices to prove that, for each prime ideal  
\( \mathfrak{p} \mid \k \), the set of \( n \in \mathbb{N} \) satisfying
\[
b(d_1^n - d_3) - a(d_2^n - d_3)
    \in \mathfrak{p}^{\, v_{\mathfrak{p}}(\k)} (d_1^n - d_3,\; d_2^n - d_3)
\]
is a finite union of arithmetic progressions.  
Fix such a prime ideal \( \mathfrak{p} \). We prove the following claim.

\begin{claim}
Suppose there doesn't exists a $ m \in \N^+$ such that $d^m_1 = d^m_2$. There exists a non-negative integer \( \ell \) such that
\[
\max_{n \ge 1} v_{\mathfrak{p}}\bigl( (d_1^n - d_3,\; d_2^n - d_3) \bigr) = \ell.
\]
Moreover, for each \( \ell_1 \in \{0,1,\dots,\ell\} \), the set
\[
\{\, n \ge 1 : v_{\mathfrak{p}}( (d_1^n - d_3,\; d_2^n - d_3) ) = \ell_1 \,\}
\]
is a finite union of arithmetic progressions.
\end{claim}

\begin{proof}[Proof of the Claim]

Note first that if one of $d_1, d_2 $ is in $\p$, then the claim is obviously true. So, we assumed that $d_1, d_2$ are not in $\p$.

We first show that
\[
\max_{n\ge 1} v_{\mathfrak{p}}\bigl( (d_1^n - d_3,\; d_2^n - d_3) \bigr)
\]
is finite. 

Assume for the purpose of contradiction that this maximum is unbounded.  
Then for every \( \ell \ge 1 \) there exists \( n \ge 1 \) such that
\begin{align}
d_1^n - d_3 &\in \mathfrak{p}^{\,\ell}, \label{eq:p1} \\
d_2^n - d_3 &\in \mathfrak{p}^{\,\ell}. \label{eq:p2}
\end{align}

Write 
\[d_1 = \mu_1 \omega_1,\]
\[d_2 = \mu_2 \omega_2,\]
where 
$\mu_1, \mu_2$ are roots of unity of order $q-1$, where $q = |\O_\p/\p|$, and $\omega_1, \omega_2 \in 1 + \p \O_\p$.

The conditions \eqref{eq:p1}-\eqref{eq:p2} become that for any $\ell \in \N$ there exists a $n \in \N$
\begin{align}
 \omega_1^{n} - \mu^{-n}_1d_3 &\in \mathfrak{p}^{\,\ell}, \label{eq:w1}\\
\omega_2^{n} -  \mu^{-n}_2d_3 &\in \mathfrak{p}^{\,\ell}. \label{eq:w2}
\end{align}
Let us denote $$S_0 \coloneqq \{ n \in \N^+ : \mu_1^n = \mu_2^n\}.$$

For any \( n \in S_0 \), we have  
\[ \mu_1^{-n} d_3 = \mu_2^{-n} d_3 \eqqcolon d_{4,n}.\] 
Thus, conditions \eqref{eq:w1}-\eqref{eq:w2} imply that for any $\ell \in \N$ large enough, there exists $n \in \N^+$ such that 
$$ \mu_1^{-n}d_3 \equiv 1 \pmod{\p},$$
$$ \mu_2^{-n}d_3 \equiv 1 \pmod{\p},$$
which implies that $\mu_1^n = \mu_2^n$.
Therefore, the condition also implies that for any large enough $l \in \N$, there exists a $n \in S_0$ such that
\[
\omega_1^{n} \equiv d_{4,n} \equiv \omega_2^{n} \pmod{\mathfrak{p}^{\,\ell}},
\]
which in particular implies that 
$$ \omega_1^{2^mn} \equiv d^{2^m}_{4,n} \equiv \omega_2^{2^mn} \pmod{\mathfrak{p}^{\,\ell}},$$
for any $m \in \N$.

Let $e =1 +  v_\p((2))$. Note that since $\omega_1$ and $\omega_2 \in 1 + \p \O_\p$, we have $\omega^{2^m}_1$ and $\omega^{2^m}_2 \in 1 + \p^{e} \O_{\p}$ for any $m \in \N$ such that $2^m \geq e$. Fixing such a $m \in \N$ and applying the \( \mathfrak{p}\)-adic logarithm yields that for any large enough $\ell \in \N$, there exists a $n \in S_0$ such that
\begin{align}
n\log_{\mathfrak{p}}(\omega^{2^m}_1) - \log_{\mathfrak{p}}(d^{2^m}_{4,n}) &\in \mathfrak{p}^{\,\ell}, \\
n\log_{\mathfrak{p}}(\omega^{2^m}_2) - \log_{\mathfrak{p}}(d^{2^m}_{4,n}) &\in \mathfrak{p}^{\,\ell}.
\end{align}
Note that this in particular implies that $v_\p(n) \leq v_\p(\log_\p d^{2^m}_{4,n})$ when $\ell$ is sufficiently large. Then in particular, for any $\ell$ large enough, there exists a $n \in S_0$ such that
\begin{align}
n\log_{\mathfrak{p}}(\omega^{2^m}_1)\log_{\mathfrak{p}}(\omega^{2^m}_2) - \log_{\mathfrak{p}}(d^{2^m}_{4,n})\log_{\mathfrak{p}}(\omega^{2^m}_2) &\in \mathfrak{p}^{\,\ell}, \\
n\log_{\mathfrak{p}}(\omega^{2^m}_2) \log_{\mathfrak{p}}(\omega^{2^m}_1)- \log_{\mathfrak{p}}(d^{2^m}_{4,n})\log_{\mathfrak{p}}(\omega^{2^m}_1) &\in \mathfrak{p}^{\,\ell},
\end{align}
which implies that 

\[
\log_{\mathfrak{p}}(d^{2^m}_{4,n})\,\bigl( \log_{\mathfrak{p}}(\omega^{2^m}_1)-\log_{\mathfrak{p}}(\omega^{2^m}_2)\bigr)
\in \mathfrak{p}^{\,\ell}
.\]
 Since $d_{4,n}$ takes finitely many values when $n$ varies in $S_0$ and $d_3$ is not a root unity, this forces
\[
\log_{\mathfrak{p}}(\omega^{2^m}_1)=\log_{\mathfrak{p}}(\omega^{2^m}_2),
\]
hence \( \omega^{2^m}_1=\omega^{2^m}_2 \).  
Since also \( \mu_1^n=\mu_2^n \) for all \(n \in S_0\), this implies  
\( d^{{2^m}n}_1=d^{{2^m}n}_2 \) for all $n \in S_0$, contradicting our assumptions.  
Thus, there exists a positive integer \( \ell \) such that
\[
\max_{n \ge 1} v_{\mathfrak{p}}\bigl( (d_1^n - d_3,\; d_2^n - d_3) \bigr) = \ell.
\]

Now fix \( \ell_1\in \{0,\dots,\ell\} \).  

If \( \ell_1=0 \), the set  
\[
\{n: v_{\mathfrak{p}}((d_1^n-d_3,d_2^n-d_3))=0\}
\]
is the complement of \( S_0 \), hence again a finite union of arithmetic progressions.

If \( \ell_1\ge 1 \), the condition  
\( v_{\mathfrak{p}}((d_1^n-d_3,d_2^n-d_3))\ge \ell_1 \)  
is equivalent to the pair of congruences
\[
d^n_1 \equiv d_3
    \pmod{\mathfrak{p}^{\,\ell_1}},\qquad
d^n_2 \equiv d_3
    \pmod{\mathfrak{p}^{\,\ell_1}},
\]
each defining a finite union of arithmetic progressions.  
Their intersection is again a finite union of arithmetic progressions.  
Subtracting the case \( \ell_1+1 \) shows that the set where the valuation is exactly \( \ell_1 \) is also a finite union of arithmetic progressions.
\end{proof}
Let  
\[
S_{\ell_1} \coloneqq 
\{\, n \ge 1 : v_{\mathfrak{p}}((d_1^n-d_3,d_2^n-d_3))=\ell_1 \,\}.
\]
Each \( S_{\ell_1} \) is a finite union of arithmetic progressions.  

Then the set of \( n \ge 1 \) with
\[
b(d_1^n - d_3) - a(d_2^n - d_3)
    \in \mathfrak{p}^{\,v_{\mathfrak{p}}(\k)} (d_1^n - d_3,\; d_2^n - d_3)
\]
is
\[
\bigcup_{0 \le \ell_1 \le \ell}
\bigl\{
n\in S_{\ell_1}:
b(d_1^n - d_3) - a(d_2^n - d_3) 
    \in \mathfrak{p}^{\,v_{\mathfrak{p}}(\k)+\ell_1}
\bigr\},
\]
which is again a finite union of arithmetic progressions.

This completes the case \( d_3 \neq 0 \) and also not a root of unity. \\

{\bf Case \RNum{2}:} Now, we assume $d_3$ is a root of unity. 

Again, it is enough to show that for any prime ideal $ \p \mid \k $, the set of $n \in \N$ such that 
$$ (b(d^n_1 - d_3) - a(d^n_2 -d_3)) \in \p^{v_\p(\k)} (d^n_1 -d_3, d^n_2 - d_3) $$
is a finite union of arithmetic progressions. 

Fix a prime $\p \mid \k$. If $\p \mid (d_1d_2)$ then 
$$ (d^n_1 - d_3, d^n_2 -d_3) \not\subseteq \p.$$
Then, the set of $n \in \N$ such that 
$$  (b(d^n_1 -d_3 ) - a(d^n_2 - d_3)) \in \p^{v_\p(\k)} (d^n_1  -d_3, d^n_2 -d_3) $$
is the same as the set of $n \in \N$ such that 
$$ (b(d^n_1 - d_3) - a(d^n_2 -d_3)) 
 \in \p^{v_\p(\k)},$$
which is a finite union of arithmetic progressions.

Now, we suppose that $(d_1d_2) \not\subseteq \p$. We denote $d_1 \coloneqq \mu_1 \omega_1$ and $d_2 \coloneqq \mu_2 \omega_2$, where $\mu_1$, $\mu_2$ are roots of unity and $\omega_1, \omega_2 \in 1 + \p\O_\p$. Then, we denote 
\[L_1 \coloneqq \{ n \in\N : \mu^n_1 = d_3\},\]
\[L_2 \coloneqq \{n \in \N : \mu_2^n = d_3\},\]
and note that they are both finite unions of arithmetic progressions.

Then, note that the set of $n \in \N \setminus (L_1 \cap L_2)$ such that 
$$ (b(d^n_1 - d_3) - a(d^n_2 -d_3)) \in \p^{v_\p(\k)} (d^n_1 -d_3, d^n_2 - d_3) $$
is the set 
\[\{ n\in \N \setminus (L_1 \cap L_2) :  (b(d^n_1 - d_3) - a(d^n_2 -d_3)) \in \p^{v_\p(\k)}\},\]
which is a finite union of arithmetic progressions.

Now, we left to show that the set \[ \{ n \in L_1 \cap L_2 : (b(d^n_1 - d_3) - a(d^n_2 -d_3)) \in \p^{v_\p(\k)} (d^n_1 -d_3, d^n_2 - d_3)\}\]
is also a finite union of arithmetic progressions.

Note that for any $n \in L_1 \cap L_2$, we have 
\[d^n_1 = d_3 \omega_1^n,\]
\[d^n_2 = d_3 \omega_2^n.\]
Hence, 
\begin{align}
   \{ n \in L_1 \cap L_2 : (b(d^n_1 - d_3) - a(d^n_2 -d_3)) \in \p^{v_\p(\k)} (d^n_1 -d_3, d^n_2 - d_3)\} \nonumber \\
    = \{ n \in L_1 \cap L_2 : (b(\omega^n_1 - 1) - a(\omega^n_2 -1)) \in \p^{v_\p(\k)} (\omega^n_1 -1, \omega^n_2 - 1)\}.
\end{align}
Then, it is enough to show the set 
\[ \{ n \in \N : (b(\omega^n_1 - 1) - a(\omega^n_2 -1)) \in \p^{v_\p(\k)} (\omega^n_1 -1, \omega^n_2 - 1)\}\]
is a finite union of arithmetic progressions.

Let $m_i$ denote the smallest positive integer such that 
$$ v_\p(\omega^{m_i}_i -1) > l$$
for $i \in \{1,2\}$, where $l = [K : \Q]$.
Then we have that for any $n \in \N^+$ such that $m_i | n$, 
\begin{align}
    v_\p(\omega^{n}_i-1) = v_\p(n/m_i (\omega^{m_i}_i -1) + \frac{n/m_i(n/m_i-1)}{2}(\omega^{m_i}_i -1)^2 + \cdots) \nonumber \\= v_\p(n) - v_\p(m_i) + v_\p(\omega^{m_i}_i-1).
\end{align}
On the other hand, for any $n \in \N^+$ such that there exists a $j \in \{1,2, \dots, m_i-1\}$ satisfying that $m_i \mid (n - j)$, we have 
\begin{equation}\label{eq: bounded-valuation-root-of-unity-1}
    v_\p(\omega^{n}_i-1) = v_\p(\omega^{n}_i - \omega^{(n-j)}_i + \omega^{(n-j)}_i - 1) = v_\p(\omega^{(n-j)}_i(\omega^{j}_i -1) + \omega^{(n-j)}_i -1) = v_\p(\omega^{j}_i -1).
\end{equation}

Let us denote $m = \lcm(m_1, m_2)$. Without loss of generality, we assume 
$$ v_\p(\omega^{m_1}_1 - 1) - v_\p(m_1) \leq v_\p(\omega^{m_2}_2 - 1) - v_\p(m_2).$$

Let us denote $$ L' \coloneqq \{sm : s \in \N\}.$$Then we first show that the set   
$$ \{n \in L' : b(\omega^{n}_1 - 1) - a(\omega^{n}_2 -1) \in \p^{v_\p(\k) + v_\p(n) - v_\p(m_1) + v_\p(\omega^{m_1}_1-1)}\}$$
is a finite union of arithmetic progressions. Note that there exists a $l_3 \in \N$ such that for any $l_4 \geq  l_3$, we have 
\begin{align}
    l_4v_\p(\omega^{m_1}_1-1) + v_\p\left( \binom{n/m_1}{l_4}\right) \nonumber\\= l_4v_\p(\omega^{m_1}_1-1) + v_\p(n) -v_\p(m_1) + v_\p\left( \binom{n/m_1 -1}{l_4-1}\right) - v_\p(l_4)  \nonumber\\\geq v_\p(\k) + v_\p(n) - v_\p(m_1) + v_\p(\omega^{m_1}_1-1).
\end{align}
Thus, it is equivalent to show the set 

\[
\begin{aligned}
\Bigl\{ n \in L' :\;&
b\!\left(
   \sum_{j=1}^{l_3-1}
   \binom{n/m_1}{j}
   (\omega_1^{m_1}-1)^j
\right)  \\
&\;-\;
a\!\left(
   \sum_{j=1}^{l_3-1}
   \binom{n/m_2}{j}
   (\omega_2^{m_2}-1)^j
\right)
\in
\mathfrak{p}^{\,v_{\mathfrak p}(\kappa)
  - v_{\mathfrak p}(m_1)
  + v_{\mathfrak p}(\omega_1^{m_1}-1)}
\Bigr\}.
\end{aligned}
\]

is a finite union of arithmetic progressions. This is true as 
$$b\left( \sum^{l_3-1}_{j = 1} \binom{n/m_1}{j}(\omega^{m_1}_1 -1)^j \right) - a\left( \sum^{l_3-1}_{j = 1} \binom{n/m_2}{j}(\omega^{m_2}_2 -1)^j \right) $$
is a polynomial on $n$.

Now suppose $n \notin L'$, then by Equation (\ref{eq: bounded-valuation-root-of-unity-1}) there exists a $l_2 \in \N$ such that 
$$ v_\p((\omega^{n}_1 - 1 , \omega^{n}_2 -1)) \leq l_2.$$
For each $i \in \{0,1, \dots, l_2\}$, let $$S'_i \coloneqq \{n\in L'^{c} : v_\p((\omega^{n}_1 - 1 , \omega^{n}_2 -1)) = i\}$$
which is a finite union of arithmetic progressions for each $i$. 

Then the set 
$$ S_i \coloneqq \{n \in S'_i : b(\omega^{n}_1 - 1) - a(\omega^{n}_2 -1) \in \p^{v_\p(\k) + i} \}$$
is also a finite union of arithmetic progressions for each $i \in \{0,1,\dots, l_2\}$. Thus, we have 
$$ \{ n \in L'^c : (b(\omega^n_1 - 1) - a(\omega^n_2 -1)) \in \p^{v_\p(\k)} (\omega^n_1 -1, \omega^n_2 - 1)\} = \bigcup^{l_2}_{i = 0}S_i$$
is a finite union of arithmetic progressions as well. \\

{\bf Case \RNum{3}:} Lastly, we handle the case that $d_3 = 0$. In this case 
$$ S \coloneqq \{n \in \N: bd^n_1 - ad^n_2 \in \k (d^n_1, d^n_2)\} = \bigcap_{\p \mid \k} \{n \in \N : bd^n_1 -ad^n_2  \in \p^{v_\p(\k) + v_\p((d^n_1, d^n_2))}\}.$$

Fix a $\p \mid \k$. If $(d_1,d_2) \subseteq \p$, then we can write $d_i = \mu_i\omega$, where $i \in \{1,2\}$, for a $\omega \in \p^{v_\p((d_1,d_2))}$ where at least one of $\mu_1,\mu_2$ is in $ \O^*_\p$. Then we have 
$$\{n \in \N : bd^n_1 -ad^n_2  \in \p^{v_\p(\k) + v_\p((d^n_1, d^n_2))}\} = \{n \in \N: b \mu_1^n - a\mu_2^n \in \p^{v_\p(\k)}\} $$
is a finite union of arithmetic progressions. 
If $(d_1, d_2) \not\subseteq \p$, then 
$$  \{n \in \N : bd^n_1 -ad^n_2  \in \p^{v_\p(\k) + v_\p((d^n_1, d^n_2))}\} = \{n \in \N : bd^n_1 - a d^n_2 \in \p^{v_\p(\k)}\}$$
is also a finite union of arithmetic progressions. 
Thus, we conclude that $S$ is a finite union of arithmetic progressions. 

\end{proof}

As a direct consequence, we answer this question of common solutions for a system of equations involving power maps: 

\begin{cor}\label{prop: poly-power-arith}
    Let $a$, $b$, $c_1$, $c_2$ and $d_3$ be non-negative integers and $\xi$ a root of unity of order $m$. Let $d_1$ and $d_2$ be two positive integers greater than $1$. Then the set of positive integers $n$ such that 
    $$ (\xi^ax^{d_1})^{\circ n} = \xi^{c_1} x^{d_3}  $$
   $$ (\xi^bx^{d_2})^{\circ n} = \xi^{c_2} x^{d_3}$$
    admit solutions in $\C^*$ is a finite union of arithmetic progressions.
\end{cor}
\begin{proof}
    Note that $\Orb_{d_1 x + a}(a) \subseteq \Z/m\Z$ is finite. Therefore, for any $a' \in \Z/m\Z$, the set of $n$'s such that $$(d_1 x + a)^{\circ n}(a) \equiv a' \pmod{m}$$
    is a finite union of arithmetic progressions. Similarly, for any $b' \in \Z/m\Z$, the set of $n$ such that $(d_2x + b)^{\circ n}(b) \equiv b' \pmod{m}$ is a finite union of arithmetic progressions. 

    Now, for any fixed $a', b' \in \Z /m \Z$, we will show the set of $n \in \N$ such that 
    \begin{equation}\label{eq: power-mod-m-eq-1}
         \xi^{a'} x^{d^n_1}   =  \xi^{c_1}x^{d_3}
    \end{equation}
    \begin{equation}\label{eq: power-mod-m-eq-2}
            \xi^{b'} x^{d_2^n} = \xi^{c_2}x^{d_3}
        \end{equation}
    admits solutions in $\C^*$ is a finite union of arithmetic progressions. With this claim, we can conclude that the set of $n$'s such that $$ (\xi^ax^{d_1})^{\circ n} = \xi^{c_1} x^{d_3}  $$
   $$ (\xi^bx^{d_2})^{\circ n} = \xi^{c_2} x^{d_3}$$
    admit solutions in $\C^*$ is a finite intersection of finite unions of arithmetic progressions, which is again a finite union of arithmetic progressions.

    To show the claim, we note that Equations (\ref{eq: power-mod-m-eq-1}) and (\ref{eq: power-mod-m-eq-2}) admit solution for a large enough $n$ such that $\min\{d^n_1, d^n_2\} > d_3$ is equivalent to
    \begin{equation}\label{eq: power-common-gcd-1}
        x^{d^n_1 -d_3} = \xi^{c_1-a'}
    \end{equation}
    \begin{equation}\label{eq: power-common-gcd-2}
        x^{d^n_2 - d_3} = \xi^{c_2-b'}
    \end{equation}
    
     admit common solutions in $\C^*$. By Lemma \ref{lem: power-reduce-to-gcd}, the set of $n \in \N$ such that 
     Equations (\ref{eq: power-common-gcd-1}) and (\ref{eq: power-common-gcd-2}) admit common solutions is 
     $$\{n \in \N : m \gcd(d^n_1-d_3, d^n_2 -d_3) \mid (c_1-b')(d^n_1 -d_3) - (c_2- a')(d^n_1 -d_3) \}.$$ Then Proposition \ref{prop: gcd-arithmetic-progressions} implies that this set of $n \in \N$ is a finite union of arithmetic progressions.
    
\end{proof}

\subsection{Special case of Question \ref{qu: DML-varieties} when $f$, $g$ and $c$ are polynomials}
In this subsection, we will show that the set of natural numbers $n$ such that 
$$ f^n(x)=g^n(x)=c(x)$$
admits a solution in $\A^1_\C$ is a finite union of arithmetic progressions. 

Note that when at least one of $f$ and $g$ is an automorphism, it is direct consequence of results in \cite{HT17}, Theorem \ref{thm: group-GDML} and Remark \ref{rmk: non-projective-GDML} (we will mention this again in the proof of the main Theorem in this subsection). So, for the main body of this subsection, we will discuss the case when both of $f$ and $g$ are of degree greater than $1$.

We first handle the relatively easier case when $\deg(f) = \deg(g)$.
\begin{prop}\label{prop: degree-eq-n-case}
  Suppose $f$ and $g$ are polynomials of positive degree $d > 1$ over $\C$ and $c$ is a polynomial over $\C$. The set of $n \in \N$ such that 
  $$ f^n(x) = g^n(x) = c(x)$$
  admits solutions is a finite union of arithmetic progressions.
\end{prop}
\begin{proof}
    If there are only finitely many $\lambda \in \C$ such that 
    $$ f^n(\lambda) = g^n(\lambda) = c(\lambda)$$
    for some $n \in \N$, then obviously the set of $n \in \N$ such that the equation admits solution is a finite union of arithmetic progressions. Thus, from now on, we assumed that there are infinitely many such $\lambda \in \C$ making the equation holds for some $n \in \N$.\\
    
    Now, if there exists a $m \in \N^+$ so that $f^m(x) \equiv g^m(x)$, then we immediately have $J(f) = J(g) = J$. Suppose there doesn't exist such a $m$. Then the same argument applying the equidistribution used in the proof of \cite[Propositions 4.2 and 4.3]{NZ25} will give us that $\Prep(f) = \Prep(g)$. Then by \cite[Corollary 1.3]{BD11}, we also have $J(f) = J(g) = J$. \\
    
    \textbf{Case \RNum{1}:} We first suppose that there doesn't exist $m \in \N^+$ such that $f^m = g^m$. Suppose $f$ and $g$ are not conjugated to power maps or Chebyshev polynomials. Then, by \cite{SS95}, there exists $\sigma_1 \in \Aut(J)$, where 
    $$ \Aut(J)\coloneqq \{ \sigma \in \Aut(\A^1) : \sigma (J) = J\},$$
    which is a finite group, such that 
    $$ f  = \sigma_1 \circ g.$$
    
    Then, for any $m \in \N$, we have that $$\sigma^{(d^m-1)/(d-1)}_{1} \circ g^{m} = f^m  \neq g^m ,$$
    which implies $\sigma^{(d^m-1)/(d-1)}_{1} \neq \id$. Then the assumption implies that, by the pigeonhole principle, there exists infinitely many $n \in \N$ such that $$\sigma^{(d^n-1)/(d-1)}_{1} = \sigma'_1$$ for some $\sigma'_1 \in \Aut(J)$ and, moreover, there are infinitely many distinct $x_n \in \C$ satisfying 
    $$ \sigma'_1 \circ g^{n}(x_n) =  g^{n}(x_n) = c(x_n).$$

    Since $\sigma'_1 \neq \id$, there is an unique $y \in \C$ such that $\sigma'_1(y) = y$. Hence, by the pigeonhole principle again, we have that there is a $y \in \C$ and infinitely many distinct $x_n \in \C$ such that 
    $c(x_n) = \sigma_1'(y)$. This implies $c$ is constant. 

    Then, note that the set of $n \in \N$ such that 
    $$ f^n(x) = \sigma^{(d^n-1)/(d-1)}_1 \circ g^n(x) = g^n(x)=c$$
    admits solution is exactly the set of $n \in \N$ such that
    $$ \sigma^{(d^n-1)/(d-1)}_1(y) = y = c $$
    holds.

    It is obvious that such a set of $n \in \N$ is a finite union of arithmetic progressions.\\

    Now, suppose $f$ is conjugated to a power map and so is $g$ as they share the same Julia set and preperiodic points. Then, after a change of coordinates, our assumption implies that there exist infinitely many $x_m \in \C$ such that 
    $$ \left(\xi_1x^d\right)^{\circ m}(x_m) = (x_m)^{d^m} = c(x_m)$$
    where $\xi_1$ is a root of unity. We may also suppose that for any $n \in \N^+$, 
    $$ \xi_1^{( d^n -1)/(d - 1)} \neq 1,$$ as otherwise $f^m = g^m$ for some $m \in \N^+$, contradicting our assumption.

    Then this implies that there are infinitely many $x_m \in \C$ such that $$\left(\xi_1x^d\right)^{\circ m}(x_m) = (x_m)^{d^m} = 0 = c(x_m).$$ 
    However, this is impossible as the only value of $y$ that can make $y^{d^m} = 0$ is $0$ which contradicts that there are infinitely many $x_m \in \C$ such that $(x_m)^{d^m} = 0$.\\

    \textbf{Case \RNum{2}:} Now, suppose there exists an $m \in \N^+$ such that $f^m = g^m$. We may assume that $m$ is the smallest such positive integer. Then obviously for any $$n \in \{mk : k \in \N\},$$ except potentially a single $n'$ which makes $f^{n'} = g^{n'}$ equal to a non-zero translate of $c(x)$, we have 
    $$ f^n(x) = g^n(x) = c(x)$$
    admits solutions. Note that the set $\{mk : k \in \N\}$ excluding potentially this single $n'$ is again a finite union of arithmetic progressions. \\
    
    Also, by the same reasoning of above, $J(g) = J(f) = J$ and there exists a $\sigma \in \Aut(J)$ such that 
    $$ f = \sigma \circ g.$$

    For any $l \in \{1, 2, \dots, m-1\}$, we have 
    $$ f^{nm + l}(x) = g^{nm + l}(x) = c(x)$$
    is equivalent to
    $$ \sigma_{l} \circ g^{mn + l}(x) = g^{mn + l}(x) = c(x)$$
    for all $n \in \N$, where $\sigma_{l}\in \Aut(J) $ and it is not the identity map.
    Note that for each $\sigma_{l}$ there exists a unique $y_l \in \C$ such that $\sigma(y_l) = y_l$. 
    For each such $y_l$, either there are only finitely many $x_0 \in \C$ such that $c(x_0) = \sigma_{l}(y_l) = y_l$ or $c \equiv \sigma_{l}(y_l) = y_l$ and for any $n \in \N$, there exists a $x_n \in \C$ such that $g^{mn+l}(x_n) = y_l$. The first case implies that the set 
    \begin{align*}
        S_{y_l} &\coloneqq \left\{ n \in \N : \exists x_0 \in \C : f^{nm+l}(x_0) = g^{nm+ l}(x_0) = c(x_0) = \sigma_{l}(y_l)\right\} \\
       & = \bigcup_{x_0 \in \C: c(x_0) = \sigma_l(y_l)}\{n \in \N : f^{nm+l}(x_0) = g^{nm+ l}(x_0) = c(x_0)  \}    \end{align*} 
    is a finite union of arithmetic progressions. 
    The second case implies that $S_{y_l} = \N$.

    Thus, we have for any $l \in \{1,2, \dots, m-1\}$, the set of $n \in \N$ such that 
    $$ f^{nm+l}(x) = g^{nm + l}(x) = c(x)$$
    admits solutions is a finite union of arithmetic progressions. Then $$\bigcup_{1 \leq l \leq m-1} S_{y_l}$$ is the set of $n \in \N$ that we are looking for, which is agian a finite union of arithmetic progressions. This concludes the proof in this case.
\end{proof}

The above proposition conclude the case when $\deg(f) = \deg(g)$. Now, we begin to discuss the case when $\deg(f) \neq \deg(g)$. We first note that if there are infinitely many $\lambda$ make $f^n(x) = g^n(x) = c(x)$ hold for some $n \in \N^+$, then $f$, $g$ and $c$ must be strongly related to each other in the following way:
\begin{prop}\label{prop: poly-infinite-sol-prep}
    Suppose $f$ and $g$ are polynomials of degree greater than $1$ over $\C$ and $c$ is a polynomial over $\C$. Suppose $\deg(f) \neq \deg(g)$ and there are infinitely many $\lambda$ such that
    $$ f^n(\lambda) = g^n(\lambda) = c(\lambda)$$
    for some $n \in \N^+$.  Then $\Prep(f) = \Prep(g)$ and one of the following holds:
    \begin{enumerate}
        \item $\Prep(c) = \Prep(f) = \Prep(g)$;
        \item $c$ is a linear automorphism of the Julia set of $f$;
        \item $c$ is a constant and is preperiodic under $f$ and $g$.
    \end{enumerate}
\end{prop}
\begin{proof}

    Similarly as explained in the proof of Proposition \ref{prop: degree-eq-n-case}, the assumption implies that $J = J(f) = J(g)$ and $\Prep(f) = \Prep(g)$.\\

     We first suppose that $f$ is not exceptional and so is $g$. Then, by \cite{SS95}, after a coordinate change, we have 
    $$ f(x) = \xi_1h^{k_1}$$
    $$ g(x)=  \xi_2 h^{k_2}$$
    where $k_1$ and $k_2$ are positive integers, $h = x^rR(x^s)$ for some $r \in \N$, $s \in \N^+$ and a polynomial $R \in \C[x]$ such that $J(h) = J$ and $\xi_1$, $\xi_2$ are roots of unity of order dividing $s$. 

    Now, by the assumption we have that $k_1 \neq k_2$ and without loss of generality $k_1 < k_2$. Then for any $x_m$ such that 
    $f^m(x_m) = g^m(x_m) = c(x_m)$ for some $m \in \N$,
    we have that $y_m = h^{k_1m}(x_m)$ satisfies that 
    $$ h^{m(k_2 - k_1)}(y_m) = \xi_3 y_m$$
    where $\xi_3$ is of order dividing $s$. Therefore, 
    $$ \Orb_{h^{(k_2-k_1)m}}(y_m) \subseteq \{\xi y_m : \xi \text{ is a root of unity of order dividing }s\},$$
    which in particular implies that $y_m \in \Prep(h) = \Prep(f) = \Prep(g).$ Thus, $x_m \in \Prep(h) = \Prep(f) = \Prep(g)$. 

    Thus, our assmption implies that 
    $$ C \coloneqq\{(x,c(x)) : x \in \P^1\} \subseteq \P^1 \times \P^1$$
    contains infinitely many preperiodic points of $(h,h)$. Now, if $c$ is a constant, then $c \in \Prep(h) = \Prep(f) = \Prep(g)$. From now on, suppose $c$ is not a constant. By \cite[Theorem 1.5]{GNY19}, there exist a non-linear polynomial $p$ commuting with some iterate of $h$ and a linear polynomial $L$ also commuting with some iterate of $h$ such that 
    $$ p^l(x) = L \circ p^k(c(x))$$
    for some integer $l,k \geq 0$. 
    
    Let us take an arbitrary $x_0 \in \Prep(p)$. Let $K$ be a finitely generated field over $\Q$ where $x_0$, $p$, $L$ and $c$ are all defined over $K$. Then $\Prep(p)(K)$ is a finite set by the Northcott property of the height function introduced by Moriwaki \cite{Mor00}. Then we have $p^l(x_0) \in \Prep(p)(K)$ as well. Since $L$ commutes with an iterate of $h$ and also some iterate of $h$ commutes with $p$, we have $\Prep(h) = \Prep(p)$ and $$L(\Prep(h)) = \Prep(h) = L^{-1}(\Prep(h)).$$ Therefore, $p^k(c(x_0)) \in \Prep(p)$ and so $c(x_0) \in \Prep(p)(K)$. Repeating the same argument for any $y_0 \in \Orb_{c}(x_0)$ in place of $x_0$, we have $c(y_0) \in \Prep(p)(K)$ as well. Thus $\Orb_{c}(x_0) \subseteq \Prep(p)(K)$, which is a finite set. Therefore, $x_0 \in \Prep(c)(K)$. Since we pick $x_0 \in \Prep(p)$ arbitrarily, we have that $\Prep(p) \subseteq \Prep(c)$. Now, if $\deg(c) > 1$, then \cite[Theorem 1.2]{BD11} implies $\Prep(c) = \Prep(g) = \Prep(f)$. Suppose $\deg(c) = 1$. Note that the above implies that $$c(\Prep(p)) \subseteq (\Prep(p)).$$ This implies that $c \in \Aut(J(p))$ and, since $J(h) = J$ and $p$ commutes with an iterate of $h$, we have $c \in \Aut(J)$.\\

    Now, suppose $f$ is exceptional and so is $g$. If $f$ is conjugated to a power map, then after a change of coordinate, we may assume that $f = \xi_1x^{d_1}$ and $g = \xi_2 x^{d_2}$ for some roots of unity $\xi_1$, $\xi_2$ and positive integers $d_1$ and $d_2$, where $d_1 \neq d_2$.  
    The assumption implies that there are infinitely many $\lambda$ such that $f^n(\lambda) = g^n(\lambda) = c(\lambda)$ for some $n \in \N$. Then, similarly, we have 
    $$C \coloneqq V(y - c(x))$$ contains infinitely many preperiodic points under $(x^2, y^2)$. By the Manin-Mumford conjecture \cite{MC95}, we have 
    $$ c(x) = \mu x^d$$
    for some root of unity $\mu$ and a non-negative integer $d$, which in particularly implies one of the three conditions in the statement holds.

    Now, suppose $f$ is conjugated to a Chebyshev polynomial of degree $>1$. Then after a suitable change of coordinate, we assume 
    $$ f = \epsilon_1 T_{d_1}$$
    $$ g= \epsilon_2 T_{d_2}$$
    where $\epsilon_1, \epsilon_2 \in \{\pm 1\}$ with $d_1 \neq d_2$. The assumption implies that there are infinitely many $\lambda$ such that 
    $$ f^n(\lambda) = g^n(\lambda) = c(\lambda)$$
    for some $n \in \N$. Then, we have similarly that
    $$ C \coloneqq V(y - c(x))$$ contains infinitely many preperiodic points of $(T_2,T_2)$. If $c$ is a constant, then again this implies that $c \in \Prep(f) = \Prep(g)$. Now, suppose $c$ is not a constant. Let $$\pi(x) = (x + x^{-1})/2,$$ we have 
    $$\overline{(\pi, \pi)^{-1}(C)}$$
    contains infinitely mamy preperiodic points under $(x^2,y^2)$ and by the Manin-Mumford conjecture, there exists a pair of coprime integers $(a,b)$ and a root of unity $\xi$ such that 
    $$  \overline{\{(t^a, \xi t^b) : t \in \C^*\} } \subseteq \overline{(\pi, \pi)^{-1}(C)}.$$
    Therefore, we have 
    $$ c\left(\frac{t^a + t^{-a}}{2}\right) = \frac{\xi t^b + \xi^{-1}t^{-b}}{2}.$$
    Since $a$ and $b$ are coprime, this implies that if $\deg(c) > 1$ then $a = 1$ and $\deg(c) = b$. Thus, we have 
    $$ c\left(\frac{ t + t^{-1}}{2}\right) = \frac{\xi t^b + \xi^{-1}t^{-b}}{2},$$
    which implies that $c = \pm T_b$ and $\xi \in \{\pm 1\}$.
\end{proof}

The following proposition handles the case when $c$ is a constant and lives in $\Prep(f) = \Prep(g)$.
\begin{prop}\label{prop: poly-case-c-constant}
    Let $f$ and $g$ be polynomials defined over $\C$ such that $\deg(f) \neq \deg(g)$ and $\Prep(f) = \Prep(g)$. Let $c \in \Prep(f) = \Prep(g)$. Then the set of $n \in \N$ such that 
    \begin{equation}\label{eq: poly-case-arith-progress-constant}
        f^n(x) = g^n(x) = c
    \end{equation}
    admits solution is a finite union of arithmetic progressions.
\end{prop}
\begin{proof}
    Let us first suppose $f$ is not linearly conjugated to a special polynomial. Then by \cite{SS95}, there exists a polynomial $h$ with $J = J(h) = J(g) = J(f)$ such that, after a suitable change of coordinate, 
    $$ f = \sigma_1 \circ h^{k_1},$$
    $$ g = \sigma_2 \circ h^{k_2},$$
    $$ h(x) = x^r R(x^s),$$
    where $R(x)$ is a non-constant polynomial, $s, k_1, k_2 \in \N^+$, $r \in \N$ and $$\sigma_1(x) = \xi_1x ,~\sigma_2(x) = \xi_2 x \in \Aut(J)$$ with some roots of unity $\xi_1, \xi_2$ of order dividing $s$.

    Let us denote $d \coloneqq\deg(h)$. Without loss of generality, we assume $k_2 > k_1$. For any $n \in \N^+$, the existence of a $x_n \in \C$ such that Equation (\ref{eq: poly-case-arith-progress-constant}) is satisfied is equivalent to the existence of $x_n \in \C$ satisfying the following two equations:
    \begin{equation}\label{eq: non-special-constant-1}
        h^{k_1n}(x_n) = c \xi_1^{-(d^{k_1n} - 1)/(d^{k_1} -1)},
    \end{equation}
    \begin{equation}\label{eq: non-special-constant-2}
        h^{(k_2- k_1)n}(c \xi_1^{-(d^{k_1n} - 1)/(d^{k_1} -1)}) = c \xi_2^{- (d^{k_2n}-1)/(d^{k_2} -1)}.
    \end{equation}

    Note that the Equation (\ref{eq: non-special-constant-2}) is equivalent to 
    \begin{equation}\label{eq: non-special-constant-3}
        h^{(k_2 -k_1)n}(c) = c \cdot\xi_2^{- (d^{k_2n}-1)/(d^{k_2} -1)} \xi_1^{d^{(k_2-k_1)n}(d^{k_1n}-1)/(d^{k_1} -1)}.
    \end{equation}
    Also, since $\xi_1, \xi_2$ are roots of unity of order dividing $s$, the expression
    $$ c \cdot \xi_2^{- (d^{k_2n}-1)/(d^{k_2} -1)} \xi_1^{d^{(k_2-k_1)n}(d^{k_1n}-1)/(d^{k_1} -1)}$$
    can only take on finitely many values as $n$ varies in $\N^+$. Moreover, for any $y \in \C$ that it can equal to, the set 
    $$ \left\{n \in \N^+ : c \cdot \xi_2^{- (d^{k_2n}-1)/(d^{k_2} -1)} \xi_1^{d^{(k_2-k_1)n}(d^{k_1n}-1)/(d^{k_1} -1)} = y\right\}$$
    is a finite union of arithmetic progressions. Thus, the set of $n \in \N^+$ such that Equation (\ref{eq: non-special-constant-3}), and hence also Equation (\ref{eq: non-special-constant-2}), holds is a finite union of arithmetic progressions. 

    Note that for any $n \in \N^+$, we can always find a $x_n \in \C$ such that Equations (\ref{eq: non-special-constant-1}) holds. Therefore, we conclude that the set of $n \in \N$ such that Equations (\ref{eq: non-special-constant-1}) and (\ref{eq: non-special-constant-2}) admit a common solution is a finite union of arithmetic progressions. Hence, the set of $n \in \N^+$ such that Equation (\ref{eq: poly-case-arith-progress-constant}) admits solutions is a finite union of arithmetic progressions.\\

    Now, suppose $f$ and $g$ are linearly conjugated to power maps. Denote $d_1 \coloneqq \deg(f)$ and $d_2 \coloneqq \deg(g)$. After a change of coordinate, we have 
    $$ f(x) = \xi_1 x^{d_1},$$
    $$g(x) =\xi_2 x^{d_2},$$
    and $c$ is also a root of unity or $0$, where $\xi_1, \xi_2$ are roots of unity. 
    
    Then the set of $n \in \N^+$ such that Equation (\ref{eq: poly-case-arith-progress-constant}) admits solutions is the set of $n \in \N^+$ such that there exists a $x_n \in \C$ making
    $$ (\xi_1x^{d_1})^n(x_n) = (\xi_2 x^{d_2})^n(x_n) = c $$
    hold. If $c = 0$, then we can always take $x_n = 0$ to make this holds for any $n \in \N^+$. Suppose $c$ is a root of unity. By Corollary \ref{prop: poly-power-arith}, we conclude that this set of $n \in \N^+$ is a finite union of arithmetic progressions.\\

    Lastly, suppose $f$ and $g$ are conjugated to Chebyshev polynomials. Then after a suitable change of coordinate, we have 
    $$ f = \sigma_1 \circ T_{d_1}$$
    $$ g = \sigma_2 \circ T_{d_2}$$
    and $c \in \Prep(T_{d_1})$, where $\sigma_1, \sigma_2 \in \{\pm x\}$. Let $\pi(x) = (x + x^{-1})/2$. We have that 
    $$ f \circ \pi = \pi \circ \sigma_1 \circ x^{d_1},$$
    $$ g \circ \pi = \pi \circ \sigma_2 \circ x^{d_2}.$$
    Then for a $n \in \N^+$, Equation (\ref{eq: poly-case-arith-progress-constant}) admits a solution if and only if there exists a $y \in \pi^{-1}(c)$, which is a root of unity, such that one of the following two systems of equations
    $$ (\sigma_1 x^{d_1})^n = y $$
    $$
    (\sigma_2 x^{d_2})^n = y
    $$
    or 
    $$ (\sigma_1 x^{d_1})^n = y $$
    $$
    (\sigma_2 x^{d_2})^n = y^{-1}
    $$
    admits a solution in $\C^*$. Then, again Corollary \ref{prop: poly-power-arith} helps us conclude that this set of $n \in \N^+$ is a finite union of arithmetic progressions in both cases. Thus, taking a union of these finite unions of arithmetic progressions, we again get that the set of $n \in \N^+$ we are looking for is a finite union of arithmetic progressions. 
    
\end{proof}

The following theorem concludes the case when $\deg(f) \neq \deg(g)$:
\begin{thm}\label{thm: f-g-c->1-fua}
    Let $f$, $g$ and $c$ be polynomials defined over $\C$. Suppose $\deg(f), \deg(g) > 1$. Then the set of $n \in \N$ such that 
    \begin{equation}\label{eq: poly-case-arith-progress}
        f^n(x) = g^n(x) = c(x)
    \end{equation} 
    admits solution is a finite union of arithmetic progressions.
\end{thm}
\begin{proof}
    If the set of $n$ such that Equation (\ref{eq: poly-case-arith-progress}) admits solution is finite, then obviously it is also a finite union of arithmetic progressions. Suppose this set of $n$ is infinite but the set of solutions 
    $$ S \coloneqq \{x_n : f^n(x_n) = g^n(x_n) = c(x_n), n \in \N\}$$
    is finite. Then for each $x_0 \in S$ that satisfies 
    $$f^n(x_0) = g^n(x_0) = c(x_0)$$
    for infinitely many $n \in  N$, we pick the minimum positive integer $m$ such that $f^{n +m}(x_0) = f^n(x_0) = c(x_0) $
    for some $n \in \N$. Then we have that 
    $$ I_{x_0} \coloneqq \{ n \in \N : f^n(x_0) = g^n(x_0) = c(x_0) \} = \{mn +l : l\in \N, \forall n \in \N\}.$$
    Thus the set of indices we are looking for is 
    $$ \bigcup_{x_0 \in S} I_{x_0}$$
    which is a finite union of arithmetic progressions.

    Now, suppose that $S$ is infinite. If $\deg(f) = \deg(g)$, then we can conclude the theorem by Proposition \ref{prop: degree-eq-n-case}. From now on we assume $\deg(f) \neq \deg(g)$, by Proposition \ref{prop: poly-infinite-sol-prep}, there are three cases
    \begin{enumerate}
        \item $\deg(c) > 1$ and $\Prep(f) = \Prep(g) = \Prep(c)$;
        \item $\deg(c) = 1$, $\Prep(f) = \Prep(g)$ and $c \in \Aut(J)$, where $J = J(f) = J(g)$;
        \item $c$ is a constant and $c \in \Prep(f) = \Prep(g)$.
        \end{enumerate}

    In case (3), Proposition \ref{prop: poly-case-c-constant} concludes that the set of $n \in \N$ that we are looking for is a finite union of arithmetic progressions. So from now on, we assume that $c$ is not a constant.\\

    Suppose $f$ is not linearly conjugated to a special polynomial. Then by \cite{SS95}, there exists a polynomial $h$ with $J = J(h) = J(g) = J(f)$ such that, after a suitable change of coordinate, 
    $$ f = \sigma_1 \circ h^{k_1},$$
    $$ g = \sigma_2 \circ h^{k_2}$$
    $$ c = \sigma_3 \circ  h^{k_3},$$
    where $k_1, k_2 \in \N^+$ and $k_3 \in \N$ and $\sigma_1, \sigma_2$ and $\sigma_3 \in \Aut(J)$. Note that in case (2), we have $k_3 = 0$. Then after a further suitable change of coordinate, for instance see the proof of the main theorem of \cite{SS95}, we have $h = x^r R(x^s)$ and $\sigma_1, \sigma_2, \sigma_3$ are multiplications by roots of unity of order dividing $s$, where $r, s \in \N$ and $R$ is a polynomial.
    
    If $r \neq 0$, it is obvious that $0$ is a solution of 
    $$ (\sigma_1 \circ h^{k_1}(x))^n = (\sigma_2 \circ h^{k_2}(x))^n = \sigma_3 \circ h^{k_3}(x)$$
    for all $n \in \N$. 

    Suppose $r = 0$. Our assumption implies that there are infinitely many $n \in \N$ such that 
    $$ \sigma_1 h^{k_1n}(x) = \sigma_2 h^{k_2n}(x)  =\sigma_3 h^{k_3}(x)$$
    admits solutions. Without loss of generality, we assume $k_1 > k_2$. This is equivalent to that there are infinitely many $n \in \N$ such that 
    
    \begin{equation}\label{eq: non-speical-infty-solution-1}
        h^{k_2n - k_3}(y) = \sigma_2^{-1}\sigma_3(y)
    \end{equation}
   \begin{equation}\label{eq: non-speical-infty-solution-2}
       h^{k_1n - k_3}(y) = \sigma^{-1}_1\sigma_3 (y)
   \end{equation}
    admit common solutions. Note that for any $n$ large enough so that $k_2 n > k_3$, if Equations (\ref{eq: non-speical-infty-solution-1}) and (\ref{eq: non-speical-infty-solution-2}) admit $y_n \in \C$ as a common solution, then we have that
\begin{align}\label{eq: non-speical-infty-solution-3}
   \sigma_1^{-1}\sigma_3 (y_n) =  h^{(k_1 -k_2)n}\circ h^{k_2n - k_3}(y_n) = h^{(k_1-k_2)n} (\sigma^{-1}_2 \sigma_3(y_n)) \nonumber\\
   = h^{(k_1 -k_2)n}(y_n).
\end{align}
Furthermore, Equations (\ref{eq: non-speical-infty-solution-1}), (\ref{eq: non-speical-infty-solution-2}) and (\ref{eq: non-speical-infty-solution-3}) implies that
    
    $$ \sigma^{-1}_1 \sigma_3(y_n) = h^{k_2n -k_3} \circ h^{(k_1- k_2)n}(y_n)=h^{k_2n - k_3}(\sigma^{-1}_1\sigma_3(y_n)) = h^{k_2n - k_3}(y_n) .$$ Hence,in summary, for any $n \in \N$ large enough so that $k_2n > k_3$, if Equations (\ref{eq: non-speical-infty-solution-1}) and (\ref{eq: non-speical-infty-solution-2}) admit a solution $y_n\in \C$, then we must have $\sigma_1 = \sigma_2$ unless $$y_n = 0$$ for all such $n$. 

   We fist handle the case that $\sigma_1 \neq \sigma_2$. In this case, the above implies that the only possible $x_0 \in \C$ such that 
    $$ f^n(x_0) = g^n(x_0) = c(x_0)$$
    for some positive integer $n > k_3/k_2$ are those $x_0 \in \C$ such that $c(x_0) = 0$. Thus, the set of $n \in \N$ such that $n > k_3/k_2$ and 
    $$ f^n(x) = g^n(x) = c(x)$$
    admits solutions is given by the set of $n \in \N$ such that 
    $$ h^{nk_1 - k_3}(0) = 0$$
    $$ h^{nk_2 -k_3}(0) = 0$$
    and $n > k_3/k_2$. Since the set of $n \in \N$ such that $n > k_3/k_2$ and
    $$ h^{nk_1 - k_3}(0) = 0$$
    is a finite union of arithmetic progressions and also set of $n$ such that $$ h^{nk_2 - k_3}(0)=0$$ is a finite union of arithmetic progressions as well. We have 
    that the set of $n \in \N$ such that Equation (\ref{eq: poly-case-arith-progress}) admits a common solution is a finite intersection of finite unions of arithmetic progression and so is again a finite union of arithmetic progressions.

    Suppose $\sigma_1 = \sigma_2$, we take $y_0$ to be a fixed point of $h$. Let $x_0 \in h^{-k_3}( \sigma_1 \sigma_3^{-1}y_0)$. Then we have 
    $$ \sigma_1 h^{k_1n}(x_0) = \sigma_1 h^{k_2n }(x_0) = \sigma_3 h^{k_3}(x_0) = \sigma_1 y_0$$
    for all $ n > k_3/k_2$. Thus the set of indices we are looking for is a finite union of arithmetic progressions in particular.

    Suppose $f$ is conjugated to a power map. Then after a suitable change of coordinate, we have 
    $$ f= \sigma_1 x^{d_1}$$
    $$ g = \sigma_2 x^{d_2}$$
    $$ c = \sigma_3 x^{d_3}$$
    where $d_1, d_2 >1$, $d_3 \in \N^+$ and $\sigma_1, \sigma_2, \sigma_3$ are roots of unity. Then we obviously have that 
    $$ f^n(0) = g^n(0) = c(0)$$
    holds for all $n \in \N$.

    Suppose $f$ is conjugated to a Chebyshev polynomial, then, after a suitable change of coordinate, we have 
    $$ f = \sigma_1 \circ T_{d_1}$$
    $$ g = \sigma_2 \circ T_{d_2}$$
    $$ c =  \sigma_3 \circ T_{d_3}$$
    where $d_1, d_2 > 1$, $d_3 \in \N^+$, $T_0(x) = x$ and $\sigma_1, \sigma_2, \sigma_3 \in \{\pm x\}$. Let $\pi(x) = (x+ x^{-1})/2$. We have 
    $$ f \circ \pi = \pi \circ \sigma_1 \circ x^{d_1},$$
    $$ g \circ \pi = \pi \circ \sigma_2 \circ x^{d_2},$$
    $$ c \circ \pi = \pi \circ \sigma_3 \circ  x^{d_3}.$$

    Then, the equation 
    $$ f^n(x) = g^n(x) = c(x)$$
    admits solutions in $\C$ for a $n \in \N$ if and only if one of the following systems of equations
    $$ (\sigma_1 x^{d_1})^n = (\sigma_2 x^{d_2})^n = \sigma_3x^{d_3} $$
    or 
    $$ (\sigma_1 x^{d_1})^n = \sigma_3x^{d_3} $$
    $$ (\sigma_2 x^{d_2})^n = \sigma_3x^{-d_3}$$
    admits solutions in $\C^*$. In the first case, Corollary \ref{prop: poly-power-arith} will conclude that the set of $n \in \N^+$ that it admits solution is a finite union of arithmetic progressions. 

    For the second case, since $\sigma_1, \sigma_2, \sigma_3$ are of order $2$, it is enough to show that the set of $n \in \N^+$ such that 
    $$  x^{d^n_1} = \sigma'^{-1}_1 \circ \sigma_3 x^{d_3}$$
    $$  x^{d^n_2} = \sigma'^{-1}_2 \circ \sigma_3 x^{-d_3}$$
    admit solutions for some $\sigma'_1, \sigma'_2 \in \{ \pm x\}$ is a finite union or arithmetic progressions. Equivalently, we only need to show that for some fixed $\eta_1, \eta_2 \in \{ \pm 1\}$ the set of $n \in \N^+$ such that $\min\{d^n_1, d^n_2\} > d_3$ and
    $$  x^{d^n_1 - d_3} = \eta_1$$
    $$ x^{d^n_2 + d_3} = \eta_2$$
    admit a solution is a finite union of arithmetic progression. 

    Applying Lemma \ref{lem: power-reduce-to-gcd}, we have that this is equivalent to show that for some fixed pair of $a,b \in \{0,1\}$, the set of $n \in \N^+$ such that $\min\{d^n_1, d^n_2\} > d_3$ and
    \begin{equation}\label{eq: chebyshev-diff-sign-1}
        2 \cdot \gcd (d^n_1 -d_3, d^n_2 + d_3) ~ \mid ~ b(d^n_1 - d_3) - a(d^n_2 + d_3)
    \end{equation}
    is a finite union of arithmetic progressions.

    If $d_1$ and $d_2$ have different parity, without loss of generality say $d_1$ is even and $d_2$ is odd, then when $n$ is sufficiently large, we have 
    $$ v_2(\gcd (d^n_1 -d_3, d^n_2 + d_3)) \leq v_2(d_3).$$
    If $d_1$ and $d_2$ have the same parity, then $2~\mid ~ d_1+ d_2$ and, by the Lifting-the-exponent lemma,
    $$ v_2(\gcd (d^n_1 -d_3, d^n_2 + d_3)) \leq v_2(d^n_1 + d^n_2) \leq v_2(d_1 + d_2).$$

    Thus, we have there exists a $l \in \N^+$ independent from $n$ such that 
    $$ v_2(\gcd (d^n_1 -d_3, d^n_2 + d_3)) \leq l,$$
    for all $n \in \N^+$. 

    Now, for each $l_1 \in \N$ such that $0 \leq l_1 \leq l$, we have the set of $n \in \N^+$ such that
    $$v_2(\gcd (d^n_1 -d_3, d^n_2 + d_3)) \geq l_1 $$
    is exactly the set of $n \in \N^+$ such that 
    $$ d^n_1 \equiv d_3 \pmod{2^{l_1}},$$
    $$ d^n_2 \equiv -d_3 \pmod{2^{l_1}},$$
    which is a finite union of arithmetic progressions. Since the complement of a finite union of arithmetic progressions is again a finite union of arithmetic progressions, we have the set of $n \in \N^+$, denoted as $L_{l_1}$, such that 
    $$v_2(\gcd (d^n_1 -d_3, d^n_2 + d_3)) = l_1 $$
    is a finite union of arithmetic progressions. 

    Note that for each $l_1$ as above, the set of $n \in L_{l_1}$ such that the expression (\ref{eq: chebyshev-diff-sign-1}) holds is the set 
    $$ L'_{l_1} \coloneqq \{n \in L_{l_1} : b(d^n_1 - d_3) - a(d^n_2 + d_3) \equiv 0 \pmod{2^{l_1 + 1}}\}$$
    which is again a finite union of arithmetic progressions.

    Note that the set of $n \in \N^+$ we are looking for is different from $\bigcup_{0 \leq l_1 \leq l} L'_{l_1}$ only by a finite set. Thus it is again a finite union of arithmetic progressions.

\end{proof}

Now, we put a summary and prove the main result of this subsection completely.

\begin{thm}[Theorem \ref{thm: main-GDML-polynomial}]
    Let $f$ and $g$ be non-constant polynomials defined over $\C$ and $c$ be a polynomial defined over $\C$. Then the set of $n \in \N$ such that 
    \begin{equation}\label{eq: poly-case-arith-progress-general}
        f^n(x) = g^n(x) = c(x)
    \end{equation} 
    admits solution is a finite union of arithmetic progressions.
\end{thm}
\begin{proof}
    Suppose $\deg(f) = \deg(g) = 1$, then this follows from Theorem \ref{thm: group-GDML} and Remark \ref{rmk: non-projective-GDML}. Suppose one of $f$ and $g$ is of degree $1$ and the other is strictly greater than $1$. Without loss of generality, we assume that $\deg(g) = 1$ and $\deg(f) > 1$. If $g$ is of finite order and there exists a $m \in \N$ such that $g^m \equiv c$, then we have that every $n \in \{mk: k \in \N^+\}$ will make Equation (\ref{eq: poly-case-arith-progress-general}) admit solutions. Also, \cite[Proposition 9]{HT17} implies that there are only finitely many $\lambda \in \C$, denote this finite set as $S$, that can make make Equation (\ref{eq: poly-case-arith-progress-general}) hold for some $n \not \in \{mk: k \in \N^+\} $. Note that for each such $\lambda \in S$, if there are infinitely many $n \in \N^+ \setminus \{mk: k \in \N^+\} $ such that 
    $$ f^n(\lambda) = g^n(\lambda) = c(\lambda) $$
    holds, then $$\lambda \in \Prep(f) \cap \Prep(g) $$
    and the set 
    $$ S_\lambda \coloneqq \{n \in \N^+ : f^n(\lambda) = g^n(\lambda) = c(\lambda)\}$$
    is a finite union of arithmetic progressions.

    Now, the set of $n \in \N^+$ such that Equation (\ref{eq: poly-case-arith-progress-general}) admits solutions is 
    $$\{mk: k \in \N^+\} \bigcup_{\lambda \in S} S_{\lambda} ,$$
    which is a finite union of arithmetic progressions.

    Suppose, on the other hand, that there doesn't exist a $m \in \N^+$ such that $c \equiv g^n$ for all $n \in \{mk: k \in \N^+\}$. Then, except for possibly two $n_0, n_1 \in \N^+$ such that $f^{n_0} \equiv c$ and $g^{n_1} \equiv c$, \cite[Proposition 9]{HT17} implies that there are only finitely many $\lambda \in \C$ such that 
    $$ f^n(\lambda) = g^n(\lambda) = c(\lambda)$$
    for some $n \in \N^+$. Then apply the above argument again, we obtain that the set of $n \in \N^+$ we are looking for is a finite union of arithmetic progressions.

    Lastly, if $\deg(f), \deg(g) > 1$, then this follows from Theorem \ref{thm: f-g-c->1-fua}.

\end{proof}

\section*{ Acknowledgements}

This project was initiated during the workshop \emph{``Arithmetic Dynamics and Diophantine Geometry''}
held at the Institute for Mathematical Sciences, National University of Singapore. 
The authors thank the organizers and the Institute for their hospitality and support. 
X.Z.\ is grateful to his advisor, Jason Bell, and Junyi Xie for many discussions. S.Y.\ thanks his advisor Junyi Xie for introducing him to this topic. S.Y.\ was supported by the National Natural Science Foundation of China (Grant No.~12271007). X.Z.\ was supported by the Natural Sciences and Engineering Research Council of Canada
 Discovery Grant (RGPIN-2022-02951).

\end{document}